\documentclass[11pt]{amsart}
\usepackage{latexsym}
\usepackage{amssymb}
\usepackage{amsfonts}
\usepackage{graphicx}
\usepackage{epsf}
\usepackage[all]{xypic}

\newtheorem{theorem}{Theorem}[section]
\newtheorem{corollary}[theorem]{Corollary}
\newtheorem{lemma}[theorem]{Lemma}
\newtheorem{proposition}[theorem]{Proposition}

\newtheorem{remark}[theorem]{Remark}

\def\bea{\begin{eqnarray*}}
\def\eea{\end{eqnarray*}}

\def\ot{\otimes}
\def\ra{\rightarrow}

\def\bea{\begin{eqnarray*}}
\def\eea{\end{eqnarray*}}

\def\dual{{\mathchar"5E}}

\begin{document}
\title{Graded Frobenius Rings}
\author{
S. D\u{a}sc\u{a}lescu$^{1}$, C. N\u{a}st\u{a}sescu$^{2}$ and L.
N\u{a}st\u{a}sescu$^{2}$}
\address{$^1$ University of Bucharest, Faculty of Mathematics and Computer Science,
Str. Academiei 14, Bucharest 1, RO-010014, Romania} \address{ $^2$
Institute of Mathematics of the Romanian Academy, PO-Box
1-764\\
RO-014700, Bucharest, Romania}
\address{
 e-mail: sdascal@fmi.unibuc.ro, Constantin\_nastasescu@yahoo.com,
 lauranastasescu@gmail.com
}

\date{}
\maketitle

\begin{abstract}
In order to study graded Frobenius algebras from a ring theoretical
perspective, we introduce graded quasi-Frobenius rings, graded
Frobenius rings and a shift-version of the latter ones, and we
investigate the structure and representations of such objects. We
need to revisit graded simple graded left Artinian rings, graded
semisimple rings, and to provide graded versions of certain results
concerning the Jacobson radical, the singular radical, and their
connection to finiteness conditions and injectivity. We prove a
structure result for (shift-)graded Frobenius rings.
\end{abstract}

\section{Introduction and preliminaries}

Originated in the work of Frobenius on group representations,
Frobenius algebras and their relatives, quasi-Frobenius algebras,
have been objects of intense study after the influential work of
Brauer, Nesbitt and Nakayama around 1940. The initial interest was
algebraic, but Frobenius algebras occurred, sometimes unexpectedly,
in topology, differential geometry, knot theory, homological
algebra, topological quantum field theory, Hopf algebra theory, etc.
A step towards a deeper understanding of Frobenius algebras from a
ring theoretical perspective was the study of (quasi-)Frobenius
rings. A presentation of the basic theory of (quasi-)Frobenius rings
and their connection to (quasi-)Frobenius algebras can be found in
\cite{lam2}.

There are certain Frobenius algebras equipped with more structure
that occur in a natural way, for example Frobenius algebras
endowed with a grading. Inspired by an equivalent characterization
of Frobenius algebras in \cite{abrams2}, one can consider
Frobenius algebras in an arbitrary monoidal category as algebras
$A$, endowed with a coalgebra structure whose comultiplication is
a morphism of $A$-bimodules. In particular, one can look at
Frobenius algebras in the monoidal category of $G$-graded vector
spaces, where $G$ is a group; these are called graded Frobenius
algebras, and they were investigated in \cite{dnn1}, as well as a
version modified by a shift, called $\sigma$-graded Frobenius
algebras. Such objects occur in noncommutative geometry, where
certain connected graded algebras are $n$-graded Frobenius for a
positive integer $n$. For example, if $A$ is a connected
Noetherian graded algebra which is Artin-Schelter regular and
Koszul, of global dimension $n$, then the Koszul dual algebra
$A^!$ of $A$ is $n$-graded Frobenius, see \cite{smith}.  Following
the point of view that Calabi-Yau algebras are related to
non-commutative potentials, see \cite{ginzburg}, it is showed in
\cite{hx} that $n$-graded Frobenius connected algebras generated
in degree 1 can be constructed from twisted superpotentials. The
structure and representation theory of graded Frobenius algebras
have been used to classification results for certain algebras
playing a role in non-commutative geometry \cite{lpwz}, and for
proving a non-commutative Bernstein-Gelfand-Gelfand correspondence
\cite{jor}.

In order to understand the structure of graded Frobenius algebras,
our initial aim was to fill in a missing piece of the Frobenius
puzzle, by defining and investigating graded quasi-Frobenius
algebras. In developing the theory, we realized that it is
interesting to consider ring theoretical versions of the concepts.
The aim of the paper is to introduce graded quasi-Frobenius rings
and ($\sigma$-)graded Frobenius rings, and to investigate them and
their representations. As expected, a finite dimensional
 graded algebra turns out to be ($\sigma$-)graded Frobenius if and only if it
is ($\sigma$-)graded Frobenius as a ring.

Some results about graded rings and graded modules may give the
impression that graded theory is a simple extension of the
un-graded one. This is true up to a point, and a reason is that
the category of modules over a ring and the category of graded
$R$-modules over a graded ring $R$ are both Grothendieck
categories. However, the category of graded $R$-modules is
equipped with a family of category isomorphisms, the shifts by
group elements, and this adds an extra level of complexity to the
structure of this category and its objects. As an example in
support of this idea, we mention the theory of the graded
Grothendieck group of an algebra graded by an abelian group,
developed in \cite{haz}. On the other hand, even in the case where
the category of graded $R$-modules is equivalent to the category
of modules over a ring $A$, this ring has usually a much more
complicated structure than $R$. For example in the case where the
grading group is finite, $A$ is the smash product $R\# G^*$, see
\cite[Chapter 7]{nvo}.

In Section \ref{sectiongradedsimple} we discuss the structure of a
graded ring $A=M_n(\Delta)(g_1,\ldots,g_n)$ associated with a graded
division ring $\Delta$, and some group elements $g_1,\ldots,g_n$.
$A$ is graded simple and graded Artinian, so any two graded simple
left $A$-modules are isomorphic up to a shift. We count the
isomorphism types of graded simple left $A$-modules, and how many of
them embed into $A$. In Section \ref{sectionsingular} we consider
the graded versions of the Jacobson radical and the singular
radical, and we prove some of their properties related to finiteness
conditions and to injectivity. We also give an alternative proof of
the structure theorem for graded simple graded left Artinian rings,
which says that any such ring is isomorphic to
$M_n(\Delta)(g_1,\ldots,g_n)$ for some $n,\Delta$ and
$g_1,\ldots,g_n$. In Section \ref{sectionprojectiveobjects} we
consider the decomposition of a graded left Artinian ring into a sum
of graded indecomposable left modules, and obtain some consequences
on the graded simple modules when we factor by the graded Jacobson
radical. A structure result for projective objects in the category
of graded modules is derived. In Section \ref{sectionQF} we define
graded quasi-Frobenius rings by proving several equivalent
characterizations. In the case of a graded ring $R$ of finite
support, we show that $R$ is graded quasi-Frobenius if and only if
it is quasi-Frobenius. More properties of graded quasi-Frobenius
rings are investigated in Section \ref{sectionmoreQF}, where we also
associate a certain set of data
 with a
graded quasi-Frobenius ring, including a version of the Nakayama
permutation. This set of data is used in Section
\ref{sectionFrobenius} to introduce graded Frobenius rings and to
give equivalent characterizations. In fact, we define the more
general version of a $\sigma$-graded Frobenius ring, which matches
with the shift-modified version of graded Frobenius algebra
mentioned above. At this point it will be clear that there is a
higher degree of complexity of the concept, compared to the
un-graded one. In the un-graded case, the Nakayama permutation and
the multiplicities of the isomorphism types of principal
indecomposable modules is all that we need for deciding whether a
quasi-Frobenius ring is Frobenius, while in the graded case it
turns out that one needs more information, related to the inertia
groups of the graded simple modules and certain shifts. We note
that for developing the theory of graded (quasi-)Frobenius rings,
we need many times to work not with isomorphism types of graded
modules, but with isoshift types, see the definition below.

Let $G$ be a group with neutral element $\varepsilon$. A ring $R$ is
$G$-graded if it has a decomposition $R=\oplus_{g\in G}R_g$ as a
direct sum of additive subgroups such that $R_gR_h\subset R_{gh}$
for any $g,h\in G$; in particular, $R_\varepsilon$ is a subring of
$R$. A graded left $R$-module is a left $R$-module $M$ with a
decomposition $M=\oplus_{g\in G}M_g$ of additive subgroups, such
that $R_gM_h\subset M_{gh}$ for any $g,h\in G$. We consider the
category $R-gr$ of graded left $R$-modules, where a morphism $f:M\ra
N$ of graded $R$-modules is an $R$-module morphism such that
$f(M_g)\subset N_g$ for any $g\in G$. If $M\in R-gr$ and $\sigma \in
G$, the $\sigma$-shift of $M$ is the graded $R$-module $M(\sigma)$
which coincides with $M$ as an $R$-module, and has the grading given
by $M(\sigma)_g=M_{g\sigma}$ for any $g\in G$. If $M,N\in R-gr$ and
$\sigma \in G$, a morphism of degree $\sigma$ from $M$ to $N$ is a
morphism $f:M\ra N$ of $R$-modules such that $f(M_g)\subset
N_{g\sigma}$ for any $g\in G$, i.e., $f$ is a morphism in $R-gr$
from $M$ to $N(\sigma)$. The category $R-gr$ is a locally finite
Grothendieck category; a family of generators is
$(R(\sigma))_{\sigma\in G}$.  We consider the equivalence relation
$\sim$, which we call the isoshift equivalence, defined as follows:
if $M, N\in R-gr$, then $M\sim N$ if and only if there exists
$\sigma \in G$ such that $M$ is isomorphic to $N(\sigma)$. The
equivalence classes with respect to $\sim$ will be called the
isoshift types of graded left $R$-modules. Similarly we can define
the category $gr-R$ of graded right $R$-modules, whose objects are
right $R$-modules $M$ with a decomposition $M=\oplus_{g\in G}M_g$
such that $M_gR_h\subset M_{gh}$ for any $g,h\in G$. For such an
object and $\sigma\in G$, the $\sigma$-shift $(\sigma)M$ is defined
by $(\sigma)M_g=M_{\sigma g}$, and we can also consider isoshift
types of graded right $R$-modules.

\section{The structure of graded simple rings}
\label{sectiongradedsimple}

Let $R$ be a $G$-graded ring. If $S$ is a graded simple left
submodule of $R$, i.e., a minimal graded left ideal,  let
$\mathcal{U}$ be the sum of all graded left submodules of $R$
which are isomorphic to a shift of $S$. Then $\mathcal{U}$ is a
two-sided graded ideal of $R$. Indeed, write
$\mathcal{U}=\sum_iS_i$, where $(S_i)_i$ is the family of all
graded left submodules of $R$ isomorphic to some shift of $S$. If
$g\in G$ and  $a\in R_g$, then the map $\varphi:R\ra R$, $\varphi
(x)=xa$, is a morphism of degree $g$ of graded left modules, and
$\mathcal{U}a=\varphi (\mathcal{U})=\sum_i \varphi (S_i)$. Since
$S_i$ is graded simple, then $\varphi (S_i)=0$ or
$\varphi(S_i)\simeq S_i(g)$. Now $S_i(g)$ is also isomorphic to a
shift of $S$, so it must be one of the $S_j$'s. Hence
$\mathcal{U}a\subseteq \mathcal{U}$, so $\mathcal{U}$ is also a
right ideal of $R$. We call $\mathcal{U}$ the (left) isoshift
component of $R$ corresponding to $S$. The socle ${\rm soc}^{\rm
gr}_\ell(R)$ of the graded left $R$-module $R$ is the direct sum
of all left isoshift components. Similarly, the right isoshift
components of $R$ are graded ideals, and their direct sum is the
right graded socle ${\rm soc}^{\rm gr}_r(R)$.

A graded ring $R$ is called graded semisimple if $R$ is a sum of
minimal graded left ideals, i.e., $R={\rm soc}^{\rm gr}_\ell(R)$.
This is equivalent to the fact that the category of graded left
$R$-modules is semisimple, i.e., any graded left $R$-module is a
sum of graded simple modules. If moreover, $R$ is a sum of minimal
graded ideals such that any two of them are isomorphic up to a
shift, then $R$ is called graded simple in \cite[page 55]{nvo}; in
order to avoid confusion, we will call such an $R$ a {\it graded
simple and graded left Artinian ring}. These properties (graded
semisimple, and graded simple and graded left Artinian) turn out
to be left-right symmetric. If $R$ is graded simple and graded
left Artinian, then any two graded simple modules are isomorphic
up to a shift, so there is just one isoshift type, and the only
isoshift component is equal to the whole of $R$. If $R$ is graded
semisimple, then $R$ is the direct sum of the left isoshift
components, which are finitely many. This decomposition shows that
a graded semisimple ring is isomorphic to a finite product of
graded simple and graded Artinian rings; see \cite[Section
2.9]{nvo}. A graded semisimple ring has the same number of
isoshift types  to the left and to the right, and this is just the
number of factors in the decomposition of $R$ as a product of
graded simple and graded left Artinian rings. We note that if $R$
is graded semisimple, a graded simple left $R$-module does not
necessarily embed into $R$, however at least one of its shifts
does.

Graded left Artinian graded rings $R$ whose only two-sided ideals
are 0 and $R$ were considered in \cite[Section 2.1]{ek}. We will
explain in Section \ref{sectionsingular} that these are the same
objects as the graded simple and graded left Artinian rings
discussed above.

Let $\Delta=\oplus_{\sigma \in G}\Delta_\sigma$ be a $G$-graded
division ring, i.e., $\Delta$ is a $G$-graded ring whose all
non-zero homogeneous elements are invertible. Let $n$ be a
positive integer, and $g_1,\ldots,g_n\in G$. We consider the
$G$-graded ring $A=M_n(\Delta)(g_1,\ldots,g_n)$, which is just
$M_n(\Delta)$ as a ring, and has a $G$-grading with the
homogeneous component of degree $\sigma\in G$ given by

$$A_\sigma=\left(
\begin{array}{cccc}
\Delta_{g_1\sigma g_1^{-1}}&\Delta_{g_1\sigma g_2^{-1}}&\ldots&\Delta_{g_1\sigma g_n^{-1}}\\
\Delta_{g_2\sigma g_1^{-1}}&\Delta_{g_2\sigma g_2^{-1}}&\ldots&\Delta_{g_2\sigma g_n^{-1}}\\
\ldots&\ldots&\ldots&\ldots\\
\Delta_{g_n\sigma g_1^{-1}}&\Delta_{g_n\sigma
g_2^{-1}}&\ldots&\Delta_{g_n\sigma g_n^{-1}}
\end{array}
\right)$$

If we denote by $e_{ij}$ the usual matrix units in $A$, then
$e_{ij}$ is homogeneous of degree $g_i^{-1}g_j$ for any $i,j$.

It is proved in \cite[Corollary 4.6.7]{nvo} and in \cite[Theorem
2.6]{ek} that a graded simple and graded left Artinian ring is
necessarily isomorphic to a graded ring of the form
$A=M_n(\Delta)(g_1,\ldots,g_n)$  as above. In both cited
references the proof uses a version of the Jacobson density
theorem for graded simple modules. We will present an alternative
proof in Section \ref{sectionsingular}. It is also indicated in
\cite[page 31]{ek} that any graded ring
$A=M_n(\Delta)(g_1,\ldots,g_n)$ of this type is graded simple and
graded left Artinian. Therefore there is just one isoshift type of
graded simple left $A$-modules. We determine how many isomorphism
types of graded simple $A$-modules exist, and how many of them
embed into $A$.

For any $1\leq j\leq n$, let $\Sigma_j$ be the left $A$-module
$\tiny{\left(
\begin{array}{c}
\Delta \\ \Delta \\
\ldots \\
\Delta
\end{array}
\right)}$ with a structure of a graded left $A$-module given by
$$(\Sigma_j)_\sigma =\tiny{\left(
\begin{array}{c}
\Delta_{g_1\sigma g_j^{-1}} \\ \Delta_{g_2\sigma g_j^{-1}} \\
\ldots \\
\Delta_{g_n\sigma g_j^{-1}}
\end{array}
\right)}$$ We have that $A\simeq \oplus_{1\leq j\leq n}\Sigma_j$.

\begin{proposition}
 $\Sigma_j$ is a graded simple module for any $1\leq j\leq n$.
\end{proposition}
\begin{proof}
Let $0\neq z=\tiny{\left(
\begin{array}{c} u_1\\\ldots \\ u_n\end{array}
\right)} \in (\Sigma_j)_\sigma$, thus $u_i\in \Delta_{g_i\sigma
g_j^{-1}}$ for any $1\leq i \leq n$. Pick $i$ such that $u_i\neq
0$. Then for any $y=\tiny{\left(
\begin{array}{c} v_1\\\ldots \\ v_n\end{array}
\right)} \in \Sigma_j$ we have
$$y=(v_1u_i^{-1}e_{1i})z+(v_2u_i^{-1}e_{2i})z+\ldots
+(v_nu_i^{-1}e_{ni})z\in Az$$ so $Az=\Sigma_j$, and this shows that
$\Sigma_j$ is graded simple.
\end{proof}

We see that for any $\sigma \in G$,
$(\Sigma_j(g_1^{-1}g_j))_\sigma=(\Sigma_j)_{\sigma g_1^{-1}g_j}$,
and this has $\Delta_{g_i\sigma
g_1^{-1}g_jg_j^{-1}}=\Delta_{g_i\sigma g_1^{-1}}$ on the $i$th row,
so then $\Sigma_j(g_1^{-1}g_j)= \Sigma_1$, or
$\Sigma_j=\Sigma_1(g_j^{-1}g_1)$. As a consequence we obtain that
$$A\simeq \oplus_{1\leq j\leq n}\Sigma_1(g_j^{-1}g_1),$$
thus $A$ is a graded simple and graded left Artinian ring.

We denote by $supp(\Delta)=\{ \sigma\in G|\Delta_\sigma\neq 0\}$
the support of $\Delta$, which is a subgroup of $G$.

\begin{proposition} \label{isosigma1}
Let $\tau \in G$. Then $\Sigma_1(\tau)\simeq \Sigma_1$ if and only
if $\tau \in g_1^{-1}supp(\Delta)g_1$.
\end{proposition}
\begin{proof}
If $\Sigma$ is a graded simple $A$-module, and $x\in
\Sigma_g\setminus \{ 0\}$, then $\varphi:A(g^{-1})\ra \Sigma$,
$\varphi (a)=ax$, is a surjective morphism of graded $A$-modules,
so then $\Sigma \simeq A(g^{-1})/{\rm ann}_A(x)$. This shows that
if $\Gamma$ is another graded simple $A$-module, then
$\Sigma\simeq \Gamma$ if and only if there exists $u\in \Gamma_g$
such that ${\rm ann}_A(u)={\rm ann}_A(x)$.

We apply this fact for $\Sigma=\Sigma_1$, $x=\tiny{\left(
\begin{array}{c} 1\\ 0\\ \ldots \\ 0\end{array}
\right)}\in (\Sigma_1)_\varepsilon$, and $\Gamma=\Sigma_1(\tau)$.
Clearly, ${\rm ann}_A(x)=\tiny{\left(
\begin{array}{cccc}
0&\Delta &\ldots &\Delta\\
0&\Delta &\ldots &\Delta\\
\ldots&\ldots &\ldots &\ldots\\
0&\Delta &\ldots &\Delta
\end{array}
\right)}$, and then $\Sigma_1(\tau)\simeq \Sigma_1$ if and only if
there exists
$$u=\tiny{\left(
\begin{array}{c} u_1\\u_2\\ \ldots \\ u_n\end{array}
\right)}\in \Sigma_1(\tau)_e=(\Sigma_1)_\tau=\tiny{\left(
\begin{array}{c}
\Delta_{g_1\tau g_1^{-1}} \\ \Delta_{g_2\tau g_1^{-1}} \\
\ldots \\
\Delta_{g_n\tau g_n^{-1}}
\end{array}
\right)}$$ with ${\rm ann}_A(u)={\rm ann}_A(x)$. But this forces
$u_2,\ldots,u_n$ to be all zero, and $u_1$ to be non-zero, and so
the existence of such a $u$ is equivalent to $\Delta_{g_1\tau
g_1^{-1}}\neq 0$. This is the same with $g_1\tau g_1^{-1}\in
supp(\Delta)$, or $\tau \in g_1^{-1}supp(\Delta)g_1$.
\end{proof}

\begin{corollary}
{\rm (i)} The number of isomorphism types of graded simple left
$A$-modules is $[G:supp(\Delta)]$.\\
{\rm (ii)} There is a bijective correspondence between the set of
isomorphism types of graded simple left $A$-modules that embed
into $A$ and the set of the right $supp(\Delta)$-cosets of $G$
containing at least one $g_i$. Moreover, the multiplicity in $A$
of one of these graded simples is the number of the $g_i$'s lying
in the corresponding coset.\\
{\rm (iii)} $A$ is gr-uniform simple, i.e., all the simple graded
left submodules of $A$ are isomorphic, if and only if all
$g_1,\ldots ,g_n$ lie in the same right $supp(\Delta)$-coset of
$G$.
\end{corollary}
\begin{proof}
(i) We know that any graded simple left $A$-module is isomorphic
to $\Sigma_1(g)$ for some $g\in G$. If $g,h\in G$, then
$\Sigma_1(g)\simeq \Sigma_1(h)$ if and only if $g$ and $h$ lie in
the same left $g_1^{-1}supp(\Delta)g_1$-coset of $G$. Thus the
number of isomorphism types of graded simple left $A$-modules is
$[G:g_1^{-1}supp(\Delta)g_1]=[G:supp(\Delta)]$.\\
(ii) We have that $\Sigma_1(g_j^{-1}g_1)\simeq
\Sigma_1(g_p^{-1}g_1)$ if and only if $g_1^{-1}g_jg_p^{-1}g_1\in
g_1^{-1}supp(\Delta)g_1$, which is the same to $g_jg_p^{-1}\in
supp(\Delta)$. Now everything is clear.\\
(iii) It follows from (ii).
\end{proof}

In a similar way we can describe the graded simple right
$A$-modules. Thus for any $1\leq i\leq n$ let $\Gamma_i=(\Delta
\ldots \Delta)$, regarded as a matrix of type $1\times n$ with
entries in $\Delta$. $\Gamma_i$ is a graded right $A$-module with
action given by usual matrix multiplication, and $G$-grading given
such that the homogeneous component of degree $\sigma$ is the set
of all elements of $\Gamma_i$ with elements from
$\Delta_{g_i\sigma g_j^{-1}}$ on the $j$th spot. Then each
$\Gamma_i$ is a graded simple right $A$-module  and $A\simeq
\oplus_{1\leq i\leq n}\Gamma_i$. Moreover $\Gamma_i\simeq
(g_1^{-1}g_i)\Gamma_1$ for any $i$. Also $(\tau)\Gamma_1\simeq
\Gamma_1$ if and only if $\tau \in g_1^{-1} supp(\Delta) g_1$. As
a consequence, the number of isomorphism types of graded simple
right $A$-modules is $[G:supp(\Delta)]$, and on the other hand,
$\Gamma_i\simeq \Gamma_j$ if and only if $g_j\in supp(\Delta)g_i$,
so the number of isomorphism types of graded simple right
$A$-modules that embed into $A$  is the number of the right
$supp(\Delta)$-cosets of $G$ containing at least one $g_i$.

We note that $A$ has the same number of isomorphism types of
graded simple left modules as the number of isomorphism types of
graded simple right modules, and then the same fact is true for
any graded semisimple ring, which is a finite product of such
$A$'s.

\section{Graded Jacobson radical and graded singular radical}
\label{sectionsingular}

Let $R=\oplus_{g\in G}R_g$ be a $G$-graded ring.  The graded
Jacobson radical of $R$ is $$J^{\rm gr}(R)=\bigcap \{ I\; |\; I
\mbox{ is a maximal graded left ideal of } R\}.$$ It turns out
that $J^{\rm gr}(R)$ is a graded ideal of $R$, and the same thing
is obtained by taking the intersection of all maximal graded right
ideals of $R$. One sees that $J^{\rm gr}(R)$ is the intersection
of the annihilators of all graded simple left (right) $R$-modules.
For any $g\in G$, the homogeneous component of degree $g$ of
$J^{\rm gr}(R)$ consists of all elements $y\in R_g$ such that
$1-xy$ is invertible for any $x\in R_{g^{-1}}$. We note that
$J^{\rm gr}(R)\cap R_\varepsilon=J(R_\varepsilon)$ and $J^{\rm
gr}(R)$ is the largest graded ideal of $R$ whose intersection with
$R_\varepsilon$ is $J(R_\varepsilon)$, see \cite[Section
2.9]{nvo}. As in the un-graded case, a graded ring $R$ is graded
semisimple if and only if it is graded left Artinian and $J^{\rm
gr}(R)=0$; in particular $R/J^{\rm gr}(R)$ is graded semisimple
for any graded left Artinian ring $R$.

\begin{remark} \label{socluanulator}
For later use, we note that if $R$ is graded left Artinian, then
for any graded left $R$-module $M$, the socle ${\rm soc}^{\rm
gr}(M)$ of $M$, i.e., the sum of all graded simple submodules of
$M$, is given by ${\rm soc}^{\rm gr}(M)=\{ m\in M|J^{\rm
gr}(R)m=0\}$.
\end{remark}

If $M$ is a graded left $R$-module and $N$ is a graded submodule
of $M$, it is known that $N$ is essential in $M$ as an
$R$-submodule (i.e., $N$ intersects non-trivially any non-zero
$R$-submodule of $M$) if and only if $N\cap N'\neq 0$ for any
non-zero graded submodule $N'$ of $M$, see \cite[Proposition
2.3.5]{nvo}. This is equivalent to the fact that for any non-zero
homogeneous element $m\in M$, there exists a homogeneous $r\in R$
such that $0\neq rm\in N$. Thus we will be able to check that a
graded submodule is essential working only with homogeneous
elements.

For any $g\in G$ define
$$\mathcal{Z}^{\rm gr}(_RR)_g=\{ x\in R_g\; |\; {\rm ann}_\ell(x) \mbox{ is
essential in }_RR\},$$ which is obviously an additive subgroup of
$R_g$, and let $\mathcal{Z}^{\rm gr}(_RR)=\sum_{g\in
G}\mathcal{Z}^{\rm gr}(_RR)_g$.

If $x_g\in \mathcal{Z}^{\rm gr}(_RR)_g$ and $r_h\in R_h$, then
$r_hx_g\in \mathcal{Z}^{\rm gr}(_RR)_{hg}$. Indeed, if $p\in G$
and $a_p\in R_p\setminus {\rm ann}_\ell(r_hx_g)$, then
$a_pr_hx_g\neq 0$, so $a_pr_h\in R_{ph}\setminus {\rm
ann}_\ell(x_g)$. Since ${\rm ann}_\ell(x_g)$ is essential in
$_RR$, there exists $q\in G$ and $b_q\in R_q$ with $0\neq
b_qa_pr_h\in {\rm ann}_\ell(x_g)$. Then $b_qa_pr_hx_g=0$, so
$b_qa_p\in {\rm ann}_\ell(r_hx_g)$, and $b_qa_p\neq 0$. This shows
that ${\rm ann}_\ell(r_hx_g)$ is essential in $_RR$.

Thus $\mathcal{Z}^{\rm gr}(_RR)$ is a graded right ideal of $R$,
and it is clear that it is also a graded left ideal, since ${\rm
ann}_\ell(x)\subset {\rm ann}_\ell(xr)$ for any $x,r\in R$.

\begin{proposition} \label{singularnilpotent}
If $R$ is graded left Noetherian then $\mathcal{Z}^{\rm gr}(_RR)$
is nilpotent.
\end{proposition}
\begin{proof}
Denote $I=\mathcal{Z}^{\rm gr}(_RR)$. Since $R$ is graded left
Noetherian, there is a positive integer $m$ such that ${\rm
ann}_\ell(I^m)={\rm ann}_\ell(I^{m+1})=\ldots $. We show that
$I^m=0$. Indeed, otherwise the family of graded left ideals $\{
{\rm ann}_\ell(z)| z\notin {\rm ann}_\ell(I^m)\mbox{ and }z\mbox{
is homogeneous}\}$ has a maximal element ${\rm ann}_\ell(x_g)$ for
some $g\in G$ and $x_g\in R_g\setminus {\rm ann}_\ell(I^m)$.

If $a_h\in I\cap R_h$, then ${\rm ann}_\ell(a_h)$ is essential in
$_RR$, so ${\rm ann}_\ell(a_h)\cap Rx_g\neq 0$. Pick $0\neq
y_px_g\in {\rm ann}_\ell(a_h)$. Now $y_p\notin {\rm ann}_\ell(x_g)$
and $y_p\in {\rm ann}_\ell(x_ga_h)$, so ${\rm
ann}_\ell(x_g)\subsetneqq {\rm ann}_\ell(x_ga_h)$. The maximality of
${\rm ann}_\ell(x_g)$ shows that $x_ga_h\in {\rm ann}_\ell(I^m)$, so
$x_ga_hI^m=0$. As this happens for any homogeneous $a_h$ in $I$, we
get that $x_gI^{m+1}=0$, so then $x_g\in {\rm
ann}_\ell(I^{m+1})={\rm ann}_\ell(I^m)$, a contradiction.
\end{proof}

At this point we need a result that will be used several times in
the sequel.

\begin{theorem} \label{gradedBaer}
{\rm (The graded version of Baer's Theorem, \cite[Corollary
2.4.8]{nvo})} Let $M$ be a graded left $R$-module. Then $M$ is
graded injective, i.e., it is an injective object in the category
$R-gr$, if and only if for any graded left ideal $I$ of $R$, any
$\sigma \in G$ and any morphism $f:I\ra M$ of degree $\sigma$ of
graded left $R$-modules, there exists $m_\sigma\in M_\sigma$ such
that $f(r)=rm_\sigma$ for any $r\in I$.
\end{theorem}

We say that the graded ring $R$ is graded left injective if it is
graded injective when regarded as a left graded $R$-module. On the
other hand, $R$ is called graded von Neumann regular if for any
homogeneous element $a$ of $R$, there exists $b\in R$ (which can
be supposed to be homogeneous) such that $a=aba$.

\begin{proposition} \label{Jacobson=singular}
Let $R$ be a graded ring which is graded left injective. Then
$\mathcal{Z}^{\rm gr}(_RR)=J^{\rm gr}(R)$ and $R/J^{\rm gr}(R)$ is
graded von Neumann regular.
\end{proposition}
\begin{proof}
Let $x_g\in J^{\rm gr}(R)_g$. We prove that ${\rm ann}_\ell(x_g)$
is essential in $_RR$, and this will show that $J^{\rm
gr}(R)\subset \mathcal{Z}^{\rm gr}(_RR)$. Indeed, if $I$ is a
graded left ideal of $R$ with ${\rm ann}_\ell(x_g)\cap I=0$, then
the map $\varphi:I\ra R(g), \varphi (a)=ax_g$, is an injective
morphism of graded left $R$-modules. Let $i:I\ra R$ be the
inclusion map. Since $R$ is injective in $R-gr$, there is a
morphism $\psi:R(g)\ra R$ of graded left $R$-modules such that
$\psi\varphi =i$. Then $a=\psi \varphi (a)=\psi
(ax_g)=ax_g\psi(1)$, or $a(1-x_g\psi(1))=0$ for any $a\in I$. As
$\psi(1)\in R_{g^{-1}}$ and $x_g\in J^{\rm gr}(R)_g$, we see that
$1-x_g\psi(1)$ is invertible, so $a=0$. Thus $I=0$.

Now we show that $\mathcal{Z}^{\rm gr}(_RR)\cap
R_\varepsilon=J(R_\varepsilon)$, and the maximality of $J^{\rm
gr}(R)$ among all graded ideals $I$ with the property that $I\cap
R_e\varepsilon=J(R_\varepsilon)$ implies that $\mathcal{Z}^{\rm
gr}(_RR)\subset J^{\rm gr}(R)$. Clearly $J(R_\varepsilon)=J^{\rm
gr}(R)\cap R_\varepsilon\subset \mathcal{Z}^{\rm gr}(_RR)\cap
R_\varepsilon$. Now let $x\in \mathcal{Z}^{\rm gr}(_RR)\cap
R_\varepsilon$. Then ${\rm ann}_\ell(x)$ is essential in $_RR$,
and clearly ${\rm ann}_\ell(x)\cap {\rm ann}_\ell (1-x)=0$, so
${\rm ann}_\ell (1-x)=0$. Then the map $\varphi:R\ra R(1-x)$ is an
isomorphism in $R-gr$. Since $R$ is graded left injective, and
$\varphi^{-1}:R(1-x)\ra R$ is a morphism in $R-gr$, there exists
$y\in R_\varepsilon$ such that $\varphi^{-1}$ is the right
multiplication by $y$. Then $r=\varphi^{-1}(r(1-x))=r(1-x)y$ for
any $r\in R$, so $(1-x)y=1$. Thus $1-x$ is right invertible for
any $x$ in the ideal $\mathcal{Z}^{\rm gr}(_RR)\cap R_\varepsilon$
of $R_\varepsilon$, so $\mathcal{Z}^{\rm gr}(_RR)\cap
R_\varepsilon\subset J(R_\varepsilon)$.

Next we show that $R/J^{\rm gr}(R)$ is graded von Neumann regular.
Let $a\in R_g$, and let $H$ be a graded left ideal of $R$ maximal
with the property that $H\cap {\rm ann}_\ell(a)=0$. Then $H+{\rm
ann}_\ell (a)$ is essential in $_RR$, and the map $\varphi:H\ra
Ha, \varphi(x)=xa$, is a bijective morphism of degree $g$ of
graded left $R$-modules. If $j:H\ra R$ is the inclusion map, then
$j\varphi^{-1}:Ha\ra R$ is a morphism of degree $g^{-1}$, and the
injectivity of $R$ in $R-gr$ shows that $j\varphi^{-1}$ is the
right multiplication by some $b\in R_{g^{-1}}$. Hence
$x=\varphi^{-1}(xa)=xab$, and then $x(a-aba)=(x-xab)a=0$ for any
$x\in H$, showing that $H(a-aba)=0$. Then clearly $(H+{\rm
ann}_\ell (a))(a-aba)=0$, i.e. $H+{\rm ann}_\ell (a)\subset {\rm
ann}_\ell (a-aba)$, showing that $a-aba\in \mathcal{Z}^{\rm
gr}(_RR)_g=J^{\rm gr}(R)_g\subset J^{\rm gr}(R)$. We conclude that
$\hat{a}=\hat{a}\hat{b}\hat{a}$ in $R/J^{\rm gr}(R)$.
\end{proof}

\begin{lemma} \label{regularelement}
Let $a\in R_g$. Then $Ra$ is a direct summand as a graded left
submodule of $R$ if and only if there exists $b\in R_{g^{-1}}$ such
that $a=aba$.
\end{lemma}
\begin{proof}
If $a=aba$ for some $b\in R_{g^{-1}}$, then $ba$ is an idempotent
in $R_\varepsilon$ and $Ra=Rba$, a direct summand of the graded
left $R$-module $R$.

Conversely, if $R=Ra\oplus I$ for some graded left ideal $I$, let
$1=u+v$ with $u,v$ idempotents of degree $\varepsilon$, $u\in Ra$,
$v\in I$. Then $u=ba$ for some $b\in R_{g^{-1}}$. Since $a=au+av$,
$au=aba\in Ra$ and $av\in I$, we get that $a=au=aba$.
\end{proof}

\begin{corollary}
If $R$ is graded semisimple then it is graded von Neumann regular.
\end{corollary}

Using Lemma \ref{regularelement}, one can follow the same approach
as in the non-graded case (for example as in \cite{goodearl}) to
obtain the following.

\begin{proposition} {\rm (\cite[Proposition 1]{hazrat})}
Let $R$ be a graded von Neumann regular ring. Then any finitely
generated graded left ideal of $R$ is a direct summand of $R$ in
$R-gr$.
\end{proposition}

\begin{corollary} \label{semisimpleNoetherianregular}
A graded ring $R$ is graded semisimple if and only if it is graded
left Noetherian and graded von Neumann regular.
\end{corollary}

\begin{proposition} \label{simpleizomorfe}
Let $R$ be a graded ring such that one of the following two
conditions is satisfied:\\
{\rm (a)} Any minimal graded left ideal is a direct summand of $R$
in
$R-gr$ (in particular if $R$ is graded regular von Neumann).\\
{\rm (b)} $R$ is graded left injective.\\
Then for any minimal graded left ideals $S$ and $S'$ of $R$ such
that $S'\simeq S(g)$ for some $g\in G$, there exists a homogeneous
$r\in R_g$ such that $S'=Sr$. As a consequence, the graded ideal
generated by any minimal graded ideal of $R$ is the whole
corresponding isoshift component of $R$.
\end{proposition}
\begin{proof}
If (a) holds, let $S=Ru$ and $S'=Rv$ for some idempotents $u,v\in
R_\varepsilon$, and let $\varphi:S'\ra S$ be an isomorphism of
degree $g$. Then $a=\varphi (v)\in S_g=R_gu$, so $a=ru$ for some
$r\in R_g$. We see that $au=ru^2=ru=a$. Let $b=\varphi^{-1}(u)\in
R_{g^{-1}}$. Then
$v=\varphi^{-1}(a)=\varphi^{-1}(au)=a\varphi^{-1}(u)=ab$. On the
other hand $b=\varphi^{-1}(u)\in (Rv)_{g^{-1}}$, so $b=sv$ for
some $s\in R_{g^{-1}}$, and then $bv=sv^2=sv=b$. We obtain
$u=\varphi(b)=\varphi(bv)=b\varphi(v)=ba$. These also show that
$aba=au=a$ and $bab=bv=b$. Now $S'=Rv=Rab\supset Rbab=Sb$. Since
$S'$ is a minimal left graded ideal, we must have $Sb=S'$ or
$Sb=0$. In the first case we are done. In the second one we get
$Rbab=0$, so $Rb=0$, showing that $b=0$, which is impossible,
since $v=ab\neq 0$.

If (b) holds, let $\varphi: S\ra S'$ be an isomorphism of degree $g$
of graded left modules. We may regard $\varphi$ as a morphism of
degree $g$ from $S$ to $R$, and then the injectivity of $R$ shows
that this morphism is the right multiplication by some $r\in R_g$.
We get $S'=Im(\varphi)=Sr$.
\end{proof}

At this point we can explain why the concept of a graded simple
ring of \cite{nvo}, i.e., a graded ring which is a direct sum of
minimal graded left ideals, any two of them graded isomorphic up
to a shift, is equivalent to the one of a graded left Artinian
ring whose only graded two-sided ideals are 0 and the whole ring,
used in \cite{ek}. Indeed, if $R$ is isomorphic to a sum of
minimal graded left ideals, any two of them graded isomorphic up
to a shift, then it is graded semisimple, and there is just one
isoshift component, which is the whole of $R$. If $I$ is a
non-zero graded ideal of $R$, then $I$ contains a minimal graded
left ideal $S$. Then $I$ contains the whole isoshift component
associated to $S$, thus $I=R$. Obviously, $R$ is graded left
Artinian, as a finite direct sum of minimal graded left ideals.
Conversely, if the only graded simple ideals of $R$ are 0 and $R$,
and $R$ is graded left Artinian, then $J^{\rm gr}(R)$, a proper
graded ideal, is 0, so $R$ is graded semisimple. Since $R$ is the
direct sum of its isoshift components, each of them being a graded
ideal, we see that there is just one such component, so $R$ is a
sum of minimal graded left ideals, any of them isomorphic up to a
shift.

Now we can give another proof for the structure theorem for graded
simple rings which are graded left Artinian. Instead of using a
density result for graded simple modules, as it is done in the
proofs provided in \cite{ek} or \cite{nvo}, we use an argument
inspired by \cite[Theorem 3.11]{lam2}.

We first  recall some constructions with graded modules. If
$M,N\in R-gr$, for any $\sigma \in G$ we denote by ${\rm
HOM}_R(M,N)_\sigma$ the set of all morphisms of degree $\sigma$
from $M$ to $N$, and we denote ${\rm HOM}_R(M,N)=\sum_{\sigma\in
G}{\rm HOM}_R(M,N)_\sigma$, a direct sum of additive subgroups of
${\rm Hom}_R(M,N)$. In general, this sum may not be the whole
${\rm Hom}_R(M,N)$, but under certain conditions we have ${\rm
HOM}_R(M,N)={\rm Hom}_R(M,N)$, for instance if $M$ is finitely
generated (\cite[Corollary 2.4.4]{nvo}), or if both $M$ and $N$
have finite support (\cite[Corollary 2.4.5]{nvo}), in particular
in the case where the grading group $G$ is finite. If $N=M$, then
we denote ${\rm HOM}_R(M,M)$ by ${\rm END}_R(M)$, and this is a
$G$-graded ring with multiplication the opposite map composition.
Moreover, $M$ is a  graded right ${\rm END}_R(M)$-module with
action $mf=f(m)$ for any $m\in M$ and $f\in {\rm END}_R(M)$.
Similar considerations can be done for graded right modules, in
which case the multiplication of the endomorphism ring is just the
usual map composition.

\begin{theorem} \label{teoremagradedsimple}
Let $R$ be a $G$-graded ring which is graded simple and graded left
Artinian. Then there exist a graded division ring $\Delta$, a
positive integer $n$, and $g_1,\ldots,g_n\in G$ such that $R$ is
isomorphic to $M_n(\Delta)(g_1,\ldots,g_n)$.
\end{theorem}
\begin{proof}
Since $R$ is graded left Artinian, it contains a minimal graded
left ideal $V$. Then $\Delta={\rm END}_R(V)$ is a graded division
ring (the multiplication is the opposite map composition), and $V$
is a graded right $\Delta$-module. Then we can consider the graded
ring $E={\rm END}(V_{\Delta})$, and the map $\varphi:R\ra E$,
$\varphi (r)(a)=ra$ for any $r\in R$ and $a\in V$, is a morphism
of $G$-graded rings. Since $R$ is graded simple, $\varphi$ is
injective.

Now let $a\in V$, $\delta \in E$ and $v\in V_g$ for some $g\in G$.
Then the map $f_v:V\ra V, f_v(x)=xv$ lies in $\Delta_g$, and then

\bea (\delta \varphi(a))(v)&=&\delta(av)\\
&=&\delta (a\cdot f_v)\\
&=&\delta(a)\cdot f_v\\
&=&f_v(\delta(a))\\
&=&\delta(a)v\\
&=&\varphi(\delta(a))(v) \eea As this holds for any homogeneous
$v\in V$, we obtain that $\delta\varphi(a)=\varphi(\delta(a))$, so
$E\varphi(V)\subset \varphi(V)$. Now $VR$ is a non-zero graded ideal
of $R$, so $VR=R$, and then $\varphi(R)=\varphi(V)\varphi(R)$. This
shows that $E\varphi(R)=E\varphi(V)\varphi(R)\subset
\varphi(V)\varphi(R)=\varphi(R)$, so $\varphi(R)$ is a left ideal of
$E$ which contains the identity element of $E$. We conclude that
$\varphi(R)=E$, so $\varphi$ is an isomorphism.

Since $\Delta$ is a graded division ring, any graded right
$\Delta$-module is free and has a basis consisting of homogeneous
elements; moreover, any two homogeneous bases have the same
cardinality. We show that a homogeneous basis of $V$ as a graded
right $\Delta$-module is finite. Indeed, otherwise we can consider
for any $\sigma\in G$ the set $$I_\sigma=\{ \delta \in {\rm
END}(V_\Delta)_\sigma\; |\; {\rm Im}\delta \mbox{ has a finite
homogeneous basis as a graded right }\Delta-{\rm module}\},$$ and
one can easily check that $I=\oplus_{\sigma\in G}I_\sigma$ is a
non-zero proper graded ideal of $E$. This is a contradiction,
since $E\simeq R$ is graded simple. We conclude that $V$ has a
finite homogeneous basis, say with $n$ elements, so $E={\rm
END}(V_\Delta)={\rm End}(V_\Delta)\simeq M_n(\Delta)$, a ring
isomorphism. Moreover, by \cite[Proposition 2.10.5]{nvo} or
\cite[pages 30-31]{ek}, $E={\rm END}(V_\Delta)\simeq
M_n(\Delta)(g_1,\ldots,g_n)$, an isomorphism of graded rings,
where $g_1,\ldots,g_n$ are the degrees of the basis elements of
$V$.
\end{proof}

\section{Projective objects in the category of graded modules over a
graded artinian ring}  \label{sectionprojectiveobjects}

If $R$ is a graded left Artinian ring, then $R/J^{\rm gr}(R)$ is
graded semisimple, and the isomorphism types of graded simple left
(right) $R$-modules are in bijection with the the isomorphism
types of graded simple left (right) $R/J^{\rm gr}(R)$-modules.
Moreover, this bijection preserves the isoshift equivalence, i.e.,
if $S_1,S_2$ are graded simple $R$-modules and $\sigma\in G$, then
$S_1\simeq S_2(\sigma)$ as graded left $R$-modules if and only if
$S_1\simeq S_2(\sigma)$ as graded left $R/J^{\rm gr}(R)$-modules.
As a consequence, $R$ has the same number of isoshift types of
graded simple modules to the left and to the right.

The following result shows which isoshift types can be found inside
$R$.

\begin{lemma} \label{scufundaresimple}
Let $M$ be a maximal graded left ideal in the $G$-graded algebra
$R$. The following are equivalent.\\
{\rm (1)} There exists $g\in G$ such that $(R/M)(g)$ embeds into
$R$.\\
{\rm (2)} $M={\rm ann}_{\ell}(x)$ for a homogeneous element $x\in
R$.\\
{\rm (3)} ${\rm ann}_r(M)\neq 0$.\\
{\rm (4)} $M={\rm ann}_{\ell}({\rm ann}_r(M))$.
\end{lemma}
\begin{proof}
{\rm (1)}$\Rightarrow$ {\rm (2)} Let $f:(R/M)(g)\ra R$ be an
injective morphism in $R-gr$. Then $x=f(\hat{1})$ is a homogeneous
element of degree $g^{-1}$ in $R$, and since $f(\hat{r})=rx$ for any
$r\in R$, we see that ${\rm ann}_{\ell}(x)=M$. Here $\hat{r}$
denotes the class of $r$ in $R/M$.\\
{\rm (2)}$\Rightarrow$ {\rm (3)} Since ${\rm ann}_{\ell}(x)=M\neq
R$, $x$ must be non-zero. Now ${\rm ann}_r(M)$ is non-zero since it
contains $x$.\\
{\rm (3)}$\Rightarrow$ {\rm (4)} The inclusion $M\subset {\rm
ann}_{\ell}({\rm ann}_r(M))$ holds for any subset $M$ of $R$. If
this inclusion is not an equality, then ${\rm ann}_{\ell}({\rm
ann}_r(M))=R$, showing that ${\rm ann}_r(M)=0$, a contradiction.
Thus $M={\rm ann}_{\ell}({\rm ann}_r(M))$.\\
{\rm (4)}$\Rightarrow$ {\rm (1)} Since $M\neq R$, we have ${\rm
ann}_r(M)\neq 0$, so there are $g\in G$ and $0\neq x\in{\rm
ann}_r(M)_{g^{-1}}$. Then $M= {\rm ann}_{\ell}({\rm
ann}_r(M))\subset {\rm ann}_{\ell}(x)\neq R$, so $M={\rm
ann}_{\ell}(x)$, and this implies that the map $f:(R/M)(g)\ra R$,
$f(\hat{r})=rx$ for any $r\in R$, is an injective morphism in
$R-gr$.
\end{proof}

The following gives graded versions of fundamental structure
results for Artinian rings, see \cite[Corollary 2.9.7]{nvo}. The
second part is the graded version of Hopkins-Levitzki Theorem (in
a slightly more general form).

\begin{theorem}\label{HopkinsLevitzki}
{\rm (1)} Let $R$ be a graded left Artinian ring. Then $J^{\rm
gr}(R)$ is nilpotent.\\
{\rm (2)} Let $R$ be a graded ring such that $J^{\rm gr}(R)$ is
nilpotent and $R/J^{\rm gr}(R)$ is graded semisimple. Then a
graded left $R$-module $M$ is graded Noetherian of and only if $M$
is graded Artinian. In particular, a graded left Artinian ring is
graded left Noetherian.
\end{theorem}

As a first consequence, we have the following.

\begin{proposition} \label{contineminimalideal}
Let $R$ be a graded right Artinian ring. Then any non-zero graded
ideal of $R$ contains a minimal graded left ideal.
\end{proposition}
\begin{proof}
Let $I$ be a non-zero graded ideal of $R$. We show that $I\cap {\rm
ann}_r(J^{\rm gr}(R))\neq 0$. Indeed, otherwise let $x\in I\setminus
\{0\}$. Then there exists $a_1\in J^{\rm gr}(R)$ with $a_1x\neq 0$.
As $a_1x \in I$, there exists $a_2\in J^{\rm gr}(R)$ such that
$a_2a_1x\neq 0$. We continue recurrently and find $a_1,a_2,\ldots
\in J^{\rm gr}(R)$ such that $a_m\ldots a_1x\neq 0$ for any positive
integer $m$. This is in contradiction to the fact that $J^{\rm
gr}(R)$ is nilpotent.

Thus $I\cap {\rm ann}_r(J^{\rm gr}(R))$ is a non-zero graded ideal
of $R$, and $J^{\rm gr}(R)(I\cap {\rm ann}_r(J^{\rm gr}(R)))=0$, so
then $I\cap {\rm ann}_r(J^{\rm gr}(R))$ is a non-zero graded left
$R/J^{\rm gr}(R)$-module, thus a graded semisimple one, since
$R/J^{\rm gr}(R)$ is graded semisimple. We conclude that $I\cap {\rm
ann}_r(J^{\rm gr}(R))$ contains a graded simple left $R/J^{\rm
gr}(R)$-module, thus also a graded simple graded $R$-submodule. This
is obviously a minimal graded left ideal contained in $I$.
\end{proof}

In the rest of this section we follow the approach in Sections 6.2
and 6.3 in \cite{pierce}, adapted to the graded case. Let $R$ be a
$G$-graded ring which is graded left Artinian.

If $P$ is a graded left $R$-module, then $P/J^{\rm gr}(R)P$ is a
graded left $R/J^{\rm gr}(R)$-module. Let $\pi_P:P\ra P/J^{\rm
gr}(R)P$ be the natural projection. If $Q$ is another graded left
$R$-module, then for any morphism $u:P\ra Q$ in $R-gr$, there exists
a unique morphism $\overline{u}:P/J^{\rm gr}(R)P\ra Q/J^{\rm
gr}(R)Q$ in $R-gr$, and also in $R/J^{\rm gr}(R)-gr$, such that
$\pi_Q u=\overline{u}\pi_P$. This defines a linear map
$$\theta_{P,Q}:{\rm Hom}_{R-gr}(P,Q)\ra {\rm Hom}_{R/J^{\rm gr}(R)-gr}(P/J^{\rm
gr}(R)P,Q/J^{\rm gr}(R)Q),\; \theta_{P,Q}(u)=\overline{u},$$ which
is surjective in the case where $P$ is a projective object in
$R-gr$, and a ring morphism in the case where $Q=P$.

\begin{proposition}
Let $R$ be a graded left Artinian ring, and let $P$ be a projective
object in the category $R-gr$. Then ${\ {\rm Ker}}\,
\theta_{P,P}=J({\rm End}_{R-gr}(P))$, thus $\theta_{P,P}$ induces a
ring isomorphism ${\rm End}_{R-gr}(P)/J({\rm End}_{R-gr}(P))\simeq
{\rm End}_{R/J^{\rm gr}(R)}(P/J^{\rm gr}(R)P)$.
\end{proposition}
\begin{proof}
Let $u\in {\ {\rm Ker}}\, \theta_{P,P}$. Then $\overline{u}=0$, so
$u(P)\subset J^{\rm gr}(R)P$. Then $u$ is nilpotent since so is
$J^{\rm gr}(R)$. This shows that ${\ {\rm Ker}}\,
\theta_{P,P}\subset J({\rm End}_{R-gr}(P))$. Then $J({\rm
End}_{R-gr}(P)/{\ {\rm Ker}}\, \theta_{P,P})=J({\rm
End}_{R-gr}(P))/{\ {\rm Ker}}\, \theta_{P,P}$.

Now ${\rm End}_{R-gr}(P)/{\ {\rm Ker}}\, \theta_{P,P}\simeq {\rm
End}_{R/J^{\rm gr}(R)}(P/J^{\rm gr}(R)P)$, and since $P/J^{\rm
gr}(R)P$ is a semisimple graded left $R/J^{\rm gr}(R)$-module, we
have that $J({\rm End}_{R/J^{\rm gr}(R)}(P/J^{\rm gr}(R)P))=0$. We
conclude that $J({\rm End}_{R-gr}(P))/{\ {\rm Ker}}\,
\theta_{P,P}=0$, and then ${\ {\rm Ker}}\, \theta_{P,P}=J({\rm
End}_{R-gr}(P))$.
\end{proof}

\begin{corollary} \label{corisoproj}
Let $P$ and $Q$ be graded projective left modules over the graded
left Artinian ring $R$. Then $P\simeq Q$ in $R-gr$ if and only if
$P/J^{\rm gr}(R)P\simeq Q/J^{\rm gr}(R)Q$ in $R/J^{\rm gr}(R)-gr$.
\end{corollary}
\begin{proof}
If $u\in {\rm Hom}_{R-gr}(P,Q)$ is an isomorphism, then clearly
$$\overline{u}=\theta_{P,Q}(u)\in {\rm Hom}_{R/J^{\rm gr}(R)-gr}(P/J^{\rm
gr}(R)P,Q/J^{\rm gr}(R)Q)$$ is an isomorphism.

Conversely, let $\varphi \in {\rm Hom}_{R/J^{\rm
gr}(R)-gr}(P/J^{\rm gr}(R)P,Q/J^{\rm gr}(R)Q)$ be an isomorphism,
with inverse $\psi$. Since $\theta_{P,Q}$ and $\theta_{Q,P}$ are
surjective, there are $u\in {\rm Hom}_{R-gr}(P,Q)$ and $v\in {\rm
Hom}_{R-gr}(Q,P)$ such that $\varphi=\overline{u}$ and
$\psi=\overline{v}$. Since $\psi\varphi=Id$, we get
$\overline{vu}=\overline{v}\overline{u}=\psi\varphi=Id=\overline{1_P}$
(note that we used the same overline symbol in several Hom-spaces,
but there is no danger of confusion), so $1_P-vu\in {\ {\rm
Ker}}\, \theta_{P,P}=J({\rm End}_{R-gr}(P))$. Then
$vu=1_P-(1_P-vu)$ is invertible in ${\rm End}_{R-gr}(P)$, so then
$u$ has a left inverse as a graded morphism. Similarly, by
$\varphi\psi=Id$, we get that $uv$ is invertible in ${\rm
End}_{R-gr}(Q)$, and then $u$ has a right inverse as a graded
morphism. We conclude that $u$ is an isomorphism.
\end{proof}

If $R$ is a graded left Artinian ring, then we have a
decomposition $R=P_1\oplus \ldots \oplus P_n$, where
$P_1,\ldots,P_n$ are graded indecomposable left modules. By the
graded version of the Krull-Schmidt Theorem, this decomposition is
unique (up to isomorphism and permutation of the factors). The
factors $P_1,\ldots,P_n$ are called the graded principal
indecomposable left $R$-modules. We are interested not only in
their isomorphism types, but also in their isoshift types.

\begin{proposition}  \label{isotypessimples}
Let $R$ be a graded left artinian ring. Then the mapping $P\mapsto
P/J^{\rm gr}(R)P$ defines a bijective correspondence between the
isomorphism types of principal graded indecomposable left
$R$-modules and the isomorphism types of graded simple left
$R/J^{\rm gr}(R)$-submodules that embed into $R/J^{\rm gr}(R)$.
Moreover, the same mapping induces a bijective correspondence
between the isoshift types of principal graded indecomposable left
$R$-modules and the isoshift types of graded simple left $R/J^{\rm
gr}(R)$-modules, and the latter are just the isoshift types of
graded simple left $R$-modules.
\end{proposition}
\begin{proof}
Let $P$ be principal graded indecomposable left $R$-module.  Since
$R$ is graded left artinian, it is also graded left noetherian, thus
a graded $R$-module of finite length, and then so is $P$. Since $P$
is indecomposable in $R-gr$, ${\rm End}_{R-gr}(P)$ is a local ring,
thus ${\rm End}_{R/J^{\rm gr}(R)}(P/J^{\rm gr}(R)P)\simeq {\rm
End}_{R-gr}(P)/J({\rm End}_{R-gr}(P))$ is a division ring. Then
$P/J^{\rm gr}(R)P$ is a graded semisimple $R/J^{\rm gr}(R)$-module
with a division ring endomorphism ring, so it must be a graded
simple module.

Now if $R=P_1\oplus \ldots \oplus P_n$ is a decomposition with
$P_1,\ldots ,P_n$ graded indecomposable modules, we see that
$R/J^{\rm gr}(R)\simeq P_1/J^{\rm gr}(R)P_1\oplus \ldots \oplus
P_n/J^{\rm gr}(R)P_n$ is a sum of graded simple $R/J^{\rm
gr}(R)$-modules. Then any isomorphism type of a graded simple
submodule of $R/J^{\rm gr}(R)$ is isomorphic to some $P_i/J^{\rm
gr}(R)P_i$. Moreover, the correspondence $P\mapsto P/J^{\rm gr}(R)P$
is injective (as isomorphism types) by Corollary \ref{corisoproj},
and we have proved the first  bijective correspondence.

The second bijective correspondence follows immediately if we use
Corollary \ref{corisoproj} for $P$ and $Q(\sigma)$, where $P$ and
$Q$ are principal graded indecomposable left $R$-modules and
$\sigma \in G$, and the fact that any graded simple left $R/J^{\rm
gr}(R)$-module is isomorphic to a shift of a graded simple left
submodule of $R/J^{\rm gr}(R)$.
\end{proof}

\begin{theorem} \label{teoremastructuraproiective}
Let $R$ be a graded left artinian ring. Then any graded projective
left $R$-module is isomorphic to a direct sum of shifts of
principal graded indecomposable left $R$-modules, and this
representation is unique up to permutation and isomorphism of the
terms.
\end{theorem}
\begin{proof}
Let $U$ be a projective object in the category $R-gr$. Since
$R/J^{\rm gr}(R)$ is graded semisimple, the graded left $R/J^{\rm
gr}(R)$-module $U/J^{\rm gr}(R)U$ is a direct sum of graded simple
$R/J^{\rm gr}(R)$-modules, and we have seen in Proposition
\ref{isotypessimples} that any such simple is isomorphic to
$(P/J^{\rm gr}(R)P)(\sigma)$ for some principal graded
indecomposable $R$-module $P$ and some $\sigma \in G$. Thus
$$U/J^{\rm gr}(R)U\simeq \oplus_{i\in I}(P_i/J^{\rm
gr}(R)P_i)(\sigma_i)\simeq (\oplus_{i\in I} P_i(\sigma_i))/J^{\rm
gr}(R)(\oplus_{i\in I} P_i(\sigma_i))$$ for some family $(P_i)_{i\in
I}$ of principal graded indecomposable $R$-modules, and some family
$(\sigma_i)_{i\in I}$ of elements of $G$. By Corollary
\ref{corisoproj} we get $U\simeq \oplus_{i\in I} P_i(\sigma_i)$.

For the uniqueness part, if $\oplus_{i\in I}U_i\simeq \oplus_{j\in
J}V_j$, where all $U_i$'s and $V_j$'s are shifts of principal
graded indecomposable $R$-modules, we have $$(\oplus_{i\in
I}U_i)/J^{\rm gr}(R)(\oplus_{i\in I}U_i)\simeq (\oplus_{j\in
J}V_j)/J^{\rm gr}(R)(\oplus_{j\in J}V_j),$$ and then $\oplus_{i\in
I}U_i/J^{\rm gr}(R)U_i\simeq \oplus_{j\in J}V_j/J^{\rm gr}(R)V_j$
as $R/J^{\rm gr}(R)$-modules. Both sides are direct sums of graded
simple modules, so the terms of the right side are isomorphic in
pairs, up to a permutation, to the ones in the left side.  Using
again Corollary \ref{corisoproj}, we get that the family
$(U_i)_{i\in I}$ is just $(V_j)_{j\in J}$ up to a permutation (in
fact a bijection from $I$ to $J$), and isomorphisms of graded
$R$-modules.
\end{proof}

As in the un-graded case, idempotents will play a prominent role
in the study of graded Artinian rings, in particular when
investigating graded principal indecomposables. Thus, if
$R=P_1\oplus \ldots \oplus P_n$ is a decomposition of the graded
ring $R$ into a sum of graded indecomposable left modules, then
$P_1=Re_1,\ldots,P_n=Re_n$ for a complete set $e_1,\ldots,e_n$ of
primitive idempotents of $R_\varepsilon$. Moreover, in this case
$R=e_1R\oplus \ldots e_nR$ is a decomposition of $R$ into a direct
sum of graded indecomposable right modules.

In order to study the isoshift types of graded principal
indecomposables, it is useful to note that if $e$ and $f$ are
idempotents in $R_\varepsilon$ and $\sigma \in G$, then $Re\simeq
Rf(\sigma)$ as graded left $R$-modules if and only if $eR\simeq
(\sigma^{-1})(fR)$ as graded right $R$-modules, see for example
\cite[Lemma 1.2]{vas}.

\section{Graded quasi-Frobenius rings} \label{sectionQF}

We recall that a ring $R$ is called quasi-Frobenius if it
satisfies any of the following equivalent conditions (see
\cite[Theorem 15.1]{lam2}: (1) $R$ is two-sided Artinian and it
satisfies the double annihilator condition for right ideals, i.e.,
${\rm ann}_r({\rm ann}_{\ell} A)=A$ for any right ideal $A$ of
$R$, and for left ideals, i.e., ${\rm ann}_{\ell}({\rm ann}_r
U)=U$ for any left ideal $U$ of $R$; (2) $R$ is left Noetherian
and satisfies the double annihilator condition for right ideals
and for left ideals; (3) $R$ is left Noetherian and injective as a
left $R$-module; (4) $R$ is right Noetherian and injective as a
left $R$-module.

The aim of this section is to introduce graded quasi-Frobenius
rings, by proving a graded version of the above mentioned theorem.
We mainly follow the approach in the un-graded case from
\cite[Section 15A]{lam2}. Several steps are similar to the
ungraded case, however new aspects dictated by the presence of
shifts occur at some other ones.
\begin{lemma} \label{lemaanulatori}
If the graded ring $R$ is graded left injective, then the following
hold.\\
{\rm (i)} ${\rm ann}_r(U)+{\rm ann}_r(V)={\rm ann}_r(U\cap V)$ for
any graded left ideals $U$ and $V$ of $R$.\\
{\rm (ii)} ${\rm ann}_r({\rm ann}_{\ell} A)=A$ for any finitely
generated graded right ideal $A$ of $R$.
\end{lemma}
\begin{proof}
(i) Let $a$ be a homogeneous element of degree $g$ in ${\rm
ann}_r(U\cap V)$. Define a map $f:U+V\ra R$ by $f(u+v)=va$ for any
$u\in U,v\in V$. This is well defined: indeed, if $u+v=u'+v'$ with
$u,u'\in U$, $v,v'\in V$, then $u-u'=v'-v\in U\cap V$, so
$(v'-v)a=0$, therefore $va=v'a$. It is clear that $f$ is a morphism
of degree $g$ of graded left $R$-modules. Since $R$ is graded left
injective, $f$ is the right multiplication by some  $b\in R_g$. Thus
$va=ub+vb$ for any $u\in U,v\in V$. For $v=0$, this shows that
$ub=0$ for any $u\in U$, thus $b\in {\rm ann}_r(U)$, while for $u=0$
we get $v(a-b)=0$ for any $v\in V$, so $a-b\in {\rm ann}_r(V)$. Then
$a=b+(a-b)\in {\rm ann}_r(U)+{\rm ann}_r(V)$. This shows that ${\rm
ann}_r(U\cap V)\subset {\rm ann}_r(U)+{\rm ann}_r(V)$. The converse
is
obvious.\\
(ii) We first prove in the case where $A=aR$ is a cyclic right
ideal generated by a homogeneous element $a$ of degree $g$. Pick
some homogeneous element $b$ of degree $h$ in the graded right
ideal ${\rm ann}_r({\rm ann}_{\ell} A)$, and define the map
$f:Ra\ra R$ by $f(ra)=rb$ for any $r\in R$. First of all, $f$ is
well defined, since $ra=r'a$ implies that $r-r'\in {\rm
ann}_{\ell} A$, so $(r-r')b=0$. Moreover, $f$ is a morphism of
degree $g^{-1}h$ of graded left $R$-modules, and the injectivity
of $R$ shows that $f$ is the right multiplication by a homogeneous
element $c$ of degree $g^{-1}h$. Then $b=f(a)=ac$, so $b\in aR$.
Thus ${\rm ann}_r({\rm ann}_{\ell} A)\subset A$, and we have
equality since the converse always holds.

Now if $A$ is an arbitrary finitely generated graded right ideal,
let $a_1,\ldots ,a_n$ be a family of homogeneous generators of $A$.
Then \bea {\rm ann}_r({\rm ann}_{\ell} A)&=& {\rm
ann}_r(\bigcap_{1\leq
i\leq n}{\rm ann}_{\ell} (a_iR))\\
&=&\sum_{1\leq i\leq n}{\rm ann}_r({\rm ann}_{\ell} (a_iR)) \;\;\;\;
(\mbox{by (i)})\\
&=&\sum_{1\leq i\leq n} a_iR \;\;\;\; (\mbox{by the cyclic case
above})\\
&=&A \eea
\end{proof}

We say that a graded ring $R$ is graded Artinian if it is graded
left Artinian and graded right Artinian.

\begin{theorem} \label{gradedQF}
Let $R$ be a $G$-graded ring. The following
assertions are equivalent.\\
{\rm (1)} $R$ is graded Artinian and it satisfies the double
annihilator condition for graded right ideals, i.e., ${\rm
ann}_r({\rm ann}_{\ell} A)=A$ for any graded right ideal $A$ of $R$,
and for graded left ideals, i.e., ${\rm ann}_{\ell}({\rm ann}_r
U)=U$ for any graded left ideal
$U$ of $R$.\\
{\rm (2)} $R$ is graded left Noetherian and satisfies the double
annihilator condition for graded right ideals and for graded left ideals.\\
{\rm (3)} $R$ is graded left Noetherian and graded
left injective.\\
{\rm (4)} $R$ is graded right Noetherian and graded left injective.\\
\end{theorem}
\begin{proof}
${\rm (1)}\Rightarrow {\rm (2)}$ follows from Theorem
\ref{HopkinsLevitzki}.\\
${\rm (2)}\Rightarrow {\rm (3)}$   We first see that if $A$ and $B$
are left graded ideals of $R$, then \bea{\rm ann}_{\ell}({\rm
ann}_rA+{\rm ann}_rB)&=&{\rm ann}_{\ell}({\rm
ann}_r A)\cap {\rm ann}_{\ell}({\rm ann}_r B)\\
&=&A\cap B \eea so then \bea {\rm ann}_r(A\cap B)&=&{\rm ann}_r({\rm
ann}_{\ell}({\rm ann}_r A+{\rm ann}_r B))\\
&=& {\rm ann}_r A+{\rm ann}_r B\eea

We show that $R$ is injective in the category $R-gr$ by using the
graded version of Baer's Theorem. Let $I$ be a graded left ideal of
$R$, and let $f:I\ra R$ be a morphism of degree $g\in G$. Since $R$
is graded left Noetherian, we have  $I=\sum_{i=1,n}Rc_i$ for some
homogeneous elements $c_1,\ldots ,c_n\in R$. We show by induction on
$n$ that there exists $y\in R_g$ such that $f(a)=ay$ for any $a\in
I$.

For $n=1$, let $d=f(c_1)$. Then ${\rm ann}_{\ell}(c_1)d=0$ since
$rc_1=0$ implies that $rd=rf(c_1)=f(rc_1)=f(0)=0$. It follows that
$d\in {\rm ann}_r({\rm ann}_{\ell} (c_1))={\rm ann}_r({\rm
ann}_{\ell} (c_1R))=c_1R$, so $d=c_1y$ for some $y\in R$. Since
$c_1$ and $d$ are homogeneous elements and ${\rm deg}(d)={\rm
deg}(c_1)g$, we can choose $y$ to be homogeneous of degree $g$. Now
$f(rc_1)=rf(c_1)=rd=rc_1y$ for any $r\in R$, and we are done.

Assume the statement holds true for $n-1$, and we prove it for $n$.
Let $I=\sum_{i=1,n}Rc_i$ with homogeneous $c_1,\ldots ,c_n\in R$,
and let $J=\sum_{i=2,n}Rc_i$. By the induction hypothesis, the
restriction of $f$ to $J$ is the right multiplication with some
$x\in R_g$. By the case $n=1$, the restriction of $f$ to $Rc_1$ is
the right multiplication to some $y\in R_g$. Now if $a\in Rc_1\cap
J$, then $ay=f(a)$ and $ax=f(a)$, so $a(x-y)=0$. Thus $x-y\in {\rm
ann}_r(Rc_1\cap J)={\rm ann}_r(Rc_1)+{\rm ann}_r(J)$. Moreover,
$x-y$ is homogeneous of degree $g$, so then there exist $y'\in {\rm
ann}_r(Rc_1)$ and $x'\in {\rm ann}_r(J)$, both homogeneous of degree
$g$, such that $x-y=x'-y'$. We show that $f$ is the right
multiplication by $z=x-x'=y-y'$, and then we are done. Indeed, let
$c\in I$, and write $c=a+b$ with $a\in Rc_1$
and $b\in J$. Then \bea f(c)&=&f(a)+f(b)\\
&=&ay+bx\\
&=&a(y-y')+ay'+b(x-x')+bx'\\
&=&az+bz\\
&=&cz \eea Note that we used that $ay'=0$ and $bx'=0$.\\
{\rm (3)}$\Rightarrow$ {\rm (4)} Since $R$ is graded left
injective, we have by Proposition \ref{Jacobson=singular} that
$\mathcal{Z}^{\rm gr}(_RR)=J^{\rm gr}(R)$ and $R/J^{\rm gr}(R)$ is
graded von Neumann regular. As $R/J^{\rm gr}(R)$ is also graded
left Noetherian, Corollary \ref{semisimpleNoetherianregular} shows
that it is graded semisimple. By Proposition
\ref{singularnilpotent}, $\mathcal{Z}^{\rm gr}(_RR)$ is nilpotent,
and then so is $J^{\rm gr}(R)$. Now Theorem \ref{HopkinsLevitzki}
shows that any graded left Noetherian module is also graded left
Artinian; in particular $R$ is graded left Artinian.

Now we show that $R$ is graded right Noetherian. Indeed, if we
assume it is not like this, let $I_1\subsetneqq I_2 \varsubsetneqq
\ldots $ be an infinite chain of graded right ideals. Pick some
homogeneous elements $a_1\in I_1, a_2\in I_2\setminus I_1, \ldots $,
and denote $U_1=a_1R, U_2=a_1R+a_2R, \ldots$. We get an infinite
chain of finitely generated graded right ideals $U_1\subsetneqq U_2
\varsubsetneqq \ldots $. Then ${\rm ann}_l(U_1)\supset {\rm
ann}_l(U_1)\supset \ldots$ is a sequence of graded left ideals, so
it terminates since $R$ is graded left Artinian. Now ${\rm
ann}_l(U_m)={\rm ann}_l(U_{m+1})=\ldots $ implies that ${\rm
ann}_r({\rm ann}_l(U_m))={\rm ann}_r({\rm ann}_l(U_{m+1}))=\ldots $,
so by Lemma \ref{lemaanulatori} we get $U_m=U_{m+1}=\ldots $, a contradiction.\\
{\rm (4)}$\Rightarrow$ {\rm (1)} Since any graded right ideal of $R$
is finitely generated, we get from Lemma \ref{lemaanulatori} that
the double annihilator condition is satisfied for graded right
ideals.

The ascending sequence of graded ideals ${\rm ann}_r(J^{\rm
gr}(R))\subset {\rm ann}_r((J^{\rm gr}(R))^2)\subset \ldots $
terminates since $R$ is graded right Noetherian,  so ${\rm
ann}_r((J^{\rm gr}(R))^n)={\rm ann}_r((J^{\rm
gr}(R))^{n+1})=\ldots$ for some $n$. Using the double annihilator
condition for graded right ideals we see that $(J^{\rm
gr}(R))^n=(J^{\rm gr}(R))^{n+1}$. Now the graded version of
Nakayama's Lemma (see \cite[Corollary 2.9.2]{nvo}) shows that the
finitely generated graded right $R$-module $(J^{\rm gr}(R))^n$
must be zero.

As in {\rm (3)}$\Rightarrow$ {\rm (4)}, the injectivity of $R$ as
a graded left module implies that $R/J^{\rm gr}(R)$ is graded von
Neumann regular. As $R/J^{\rm gr}(R)$ is also graded right
Noetherian, the version of Corollary
\ref{semisimpleNoetherianregular} to the right shows that it is
graded semisimple. Since $R$ is graded right Noetherian, we get
that it is also graded right Artinian by the version of Theorem
\ref{HopkinsLevitzki} to the right.

Now Lemma \ref{scufundaresimple} shows that all isoshift types of
graded simple right modules can be found inside $R$, since the
double annihilator condition holds for maximal graded right
ideals. Let $\mathcal{U}_1,\ldots, \mathcal{U}_n$ be the isoshift
type components of the right graded $R$-module $R$. By Proposition
\ref{contineminimalideal} , each $\mathcal{U}_i$ contains a
minimal graded left ideal $\Sigma_i$. We claim that
$\Sigma_1,\ldots,\Sigma_n$ lie in different isoshift classes.
Indeed, if $\Sigma_i\simeq \Sigma_j(g)$ for some $i\neq j$ and
some $g\in G$, then by Proposition \ref{simpleizomorfe} we get
that $\Sigma_j=\Sigma_ir$ for some $r\in R_g$. But then
$\Sigma_j\subset \mathcal{U}_ir\subset \mathcal{U}_i$. As
$\Sigma_j\subset \mathcal{U}_j$, this provides a contradiction.
Since $R$ is graded right Artinian, the number of isoshift types
of graded simple left $R$-modules is also $n$, so then $\Sigma_1,
\ldots,\Sigma_n$ is a system of representatives for all these
isoshift types. We conclude that all isoshift types of graded
simple left modules can be found inside $R$.

Next we show that for any non-zero graded left $R$-module $M$, there
exist some $g\in G$ and a non-zero morphism of degree $g$ of graded
left $R$-modules $f:M\ra R$. Indeed, since $J^{\rm gr}(R)$ is
nilpotent, we have that $J^{\rm gr}(R)M\neq M$. Then $M/J^{\rm
gr}(R)M$ is a non-zero graded left $R/J^{\rm gr}(R)$-module.
Moreover, $M/J^{\rm gr}(R)M$ is a sum of graded simple left
$R/J^{\rm gr}(R)$-modules, since $R/J^{\rm gr}(R)$ is graded
semisimple. In particular, there is a graded simple left $R/J^{\rm
gr}(R)$-module $S$ and a surjective morphism $\phi:M\ra S$ of graded
left $R/J^{\rm gr}(R)$-modules. But this is also a surjective
morphism of graded left $R$-modules, and $S$ is also simple as a
graded $R$-module. Moreover, since the isoshift type $S$ lies inside
$R$, there is an injective morphism of degree $g$ of graded left
$R$-modules $\psi:S\ra R$ for some $g\in G$. If $\pi:M\ra M/J^{\rm
gr}(R)M$ is the natural projection, then $\psi\phi \pi:M\ra R$ is a
non-zero morphism of degree $g$ of graded left $R$-modules.

Now we show that the double annihilator condition holds for graded
left ideals. Indeed, if we assume that for a graded left ideal $U$
we have ${\rm ann}_{\ell}({\rm ann}_r U)\neq U$, let $M={\rm
ann}_{\ell}({\rm ann}_r U)/U$, a non-zero graded left $R$-module,
and let $\pi:{\rm ann}_{\ell}({\rm ann}_r U)\ra M$ be the natural
projection. We showed above that there exists a non-zero morphism
$f:M\ra R$ of degree $g$ of graded left $R$-modules for some $g\in
G$. Then $f\pi:{\rm ann}_{\ell}({\rm ann}_r U)\ra R$ is a morphism
of degree $g$ of graded left $R$-modules, and the injectivity of $R$
shows that $f\pi$ is the right multiplication by some $r\in R_g$. As
$f\pi (U)=0$, we see that $Ur=0$, so $r\in {\rm ann}_r U$. Therefore
for any $x\in {\rm ann}_{\ell}({\rm ann}_r U)$ we have
$0=xr=(f\pi)(x)$. Thus $f\pi=0$, and so $f=0$, a contradiction.

Finally, we show that $R$ is also graded left Artinian. Indeed, let
$U_1\supset U_2\supset \ldots $ be a descending chain of graded left
ideals of $R$. Then ${\rm ann}_r(U_1)\subset {\rm ann}_r(U_2)\subset
\ldots$ is an ascending chain of graded right ideals, so ${\rm
ann}_r(U_m)\subset {\rm ann}_r(U_{m+1})= \ldots$ for some $m$. Then
${\rm ann}_{\ell}({\rm ann}_r U_m)={\rm ann}_{\ell}({\rm ann}_r
U_{m+1})=\ldots$, and the double annihilator condition shows that
$U_m=U_{m+1}=\ldots$, which ends the proof.
\end{proof}

\begin{remark} \label{definitiegradedQF}
{\rm A graded ring $R$ is called {\it graded quasi-Frobenius} if
it satisfies the equivalent conditions of Theorem \ref{gradedQF}.
We note that condition {\rm (1)} in the Theorem is left-right
symmetric, so we can add more equivalent conditions to the
theorem: any condition saying that $R$ is graded Noetherian at one
side and graded injective at one side, also the condition saying
that $R$ is graded right Noetherian and it satisfies the two
double annihilator conditions.}
\end{remark}

We note that if a graded ring $R$ is quasi-Frobenius as a ring, then
it is graded quasi-Frobenius. Indeed, since $R$ is injective as a
left $R$-module, it follows that $R$ is graded left injective by
\cite[Corollary 2.3.2]{nvo}, and since $R$ is left Noetherian, we
clearly have that $R$ is graded left Noetherian. The converse does
not hold in general, thus a graded ring $R$ may be graded
quasi-Frobenius without being quasi-Frobenius. For example, let
$R=k[X,X^{-1}]$ be the Laurent polynomial ring over a field $k$,
with its usual $\mathbb{Z}$-graded ring structure. Then $R$ is a
graded division ring, so it is graded injective in view of Theorem
\ref{gradedBaer}. Obviously, it is also graded Noetherian. Thus $R$
is graded quasi-Frobenius. However, $R$ is not a quasi-Frobenius
ring, since it is not injective as an $R$-module, see for example
\cite[Remark 2.3.3]{nvo}.

In the case where the graded ring $R$ has finite support, i.e.,
only finitely many homogeneous components $R_g$ are non-zero, in
particular when $G$ is a finite group, the graded quasi-Frobenius
and quasi-Frobenius conditions on $R$ are equivalent, as the
following shows.

\begin{proposition}
Let $R$ be a graded ring of finite support. Then $R$ is graded
quasi-Frobenius if and only if it is a quasi-Frobenius ring.
\end{proposition}
\begin{proof}
If $R$ is a graded ring of finite support which is graded left
injective, then $R$ is injective as a left $R$-module by
\cite[Theorem 3.9]{dnrv}. On the other hand, if $R$ is graded left
Noetherian, then it is easy to see that each $R_g$ is a Noetherian
left $R_\varepsilon$-module. Since $R$ has finite support, we get
that $R$ is a Noetherian left $R_\varepsilon$-module, so it is
Noetherian as a left $R$-module, too.
\end{proof}

At the end of this section we list some properties of graded
quasi-Frobenius rings that follow from the proof of Theorem
\ref{gradedQF}. Thus let $R$ be a graded quasi-Frobenius ring.
Then:

$\bullet$ Any isoshift type of graded simple module to the left
and to the right can be found inside $R$. Thus for any graded
simple left (respectively right) $R$-module $\Sigma$ there exists
$g\in G$ such that $\Sigma (g)$ (respectively $(g)\Sigma$) embeds
into $R$.

$\bullet$ The number of isoshift types of graded simple left
$R$-modules is equal to the number of isoshift types of graded
simple right $R$-modules. Moreover, the left isoshift components of
$R$  coincide with the right isoshift components.

$\bullet$ As a consequence, ${\rm soc}^{\rm gr}_l(R)={\rm soc}^{\rm
gr}_r(R)$. Since $R$ is graded left  Artinian, we use Remark
\ref{socluanulator} to see that ${\rm soc}^{\rm gr}_l(R)={\rm
ann}_r(J^{\rm gr}(R))$ and ${\rm soc}^{\rm gr}_r(R)={\rm
ann}_{\ell}(J^{\rm gr}(R))$. We conclude that
$${\rm soc}^{\rm gr}_l(R)={\rm soc}^{\rm gr}_r(R)={\rm ann}_r(J^{\rm gr}(R))={\rm ann}_{\ell}(J^{\rm
gr}(R)).$$

\section{More properties of graded quasi-Frobenius rings}
\label{sectionmoreQF}

The following result is a graded version of \cite[Theorem
15.9]{lam2} and the proof is inspired by it. However, some general
results about Grothendieck categories are needed.

\begin{theorem}
Let $R$ be graded ring. The following are equivalent:\\
{\rm (1)} $R$ is graded quasi-Frobenius.\\
{\rm (2)} The classes of injective objects and projective objects
coincide in the category $R-gr$.
\end{theorem}
\begin{proof}
A key result that we need is the following graded version of
Bass-Matlis-Papp Theorem: if $R$ is a graded ring, then the
following are equivalent: (a) $R$ is graded left Noetherian; (b) Any
direct sum of injective objects in $R-gr$ is graded injective; (c)
Any injective object in $R-gr$ is a direct sum of indecomposable
injective objects in $R-gr$. This follows from the more general
result \cite[Theorem 3]{stARK}, formulated for certain Grothendick
categories.

Assume now that $R$ is graded quasi-Frobenius. If $P$ is a
projective object in $R-gr$, then $P$ is a direct summand in
$\oplus_{i\in I}R(\sigma_i)$ for some family $(\sigma_i)_{i\in I}$
of elements of $G$. Each $R(\sigma_i)$ is graded injective since so
is $R$, and then $\oplus_{i\in I}R(\sigma_i)$ is graded injective by
(a)$\Rightarrow$(b) in the general result mentioned above. We obtain
that $P$ is graded injective. On the other hand, if $Q$ is an
injective object in $R-gr$, we use (a)$\Rightarrow$(c) in the
general result to see that $Q=\oplus_{i\in I} Q_i$, a direct sum of
injective indecomposable objects in $R-gr$. Since $R$ is graded left
Artinian, each $Q_i$ contains a graded simple module $S_i$, and then
$Q_i=E^{\rm gr}(S_i)$, the injective envelope of $S_i$ in $R-gr$. On
the other hand, $S_i$ embeds into $R(\sigma_i)$ for some
$\sigma_i\in G$. As $R(\sigma_i)$ is graded injective, we get that
$Q_i$ embeds into $R(\sigma_i)$, hence it is a direct summand. This
shows that each $Q_i$ is graded projective, and then so is
$Q=\oplus_{i\in I} Q_i$.

Conversely, assume that injectives and projectives are the same in
$R-gr$. Then $R$ is injective in $R-gr$, since it is obviously
projective. On the other hand, any direct sum of injectives
(=projectives) in $R-gr$ is projective (=injective), and then $R$ is
graded left Noetherian by (b)$\Rightarrow$(a) in the general result.
We obtain that $R$ is graded quasi-Frobenius.
\end{proof}

Let $M\in R-gr$. We denote by $M\dual$  the graded right
$R$-module ${\rm HOM}_R(M,R)$, which is a submodule of ${\rm
Hom}_R(M,R)$. As we explained before Theorem
\ref{teoremagradedsimple}, $M\dual={\rm Hom}_R(M,R)$ whenever $M$
is finitely generated. Similarly, we consider the graded left
$R$-module $ N\dual={\rm HOM}_R(N,R)$ for any graded right
$R$-module $N$. The natural map $\varphi_M:M\ra
(M\dual)\dual,\varphi (m)(f)=f(m)$ for any $m\in M$ and $f\in
M\dual$, is a morphism of graded left $R$-modules. We will use the
simpler notation $M\dual\dual$ for $(M\dual)\dual$.

\begin{lemma} \label{lemadualsuspensie}
Let $M\in R-gr$ and let $\tau\in G$. Then $M(\tau)\dual=
(\tau^{-1})M\dual$.
\end{lemma}
\begin{proof}
Recall that the homogeneous component of degree $\sigma$ of $M\dual$
consists of all morphisms of $R$-modules $f:M\ra R$ such that
$f(M_g)\subset R_{g\sigma}$ for any $g\in G$. Then the homogeneous
component of degree $\sigma$ of $M(\tau)\dual$ consists of all
morphisms of $R$-modules $f:M\ra R$ such that $f(M_{g\tau})\subset
R_{g\sigma}$ for any $g\in G$, and this is equivalent to
$f(M_h)\subset R_{h\tau^{-1}\sigma}$ for any $h\in G$, which means
that $f\in (M\dual)_{\tau^{-1}\sigma}=((\tau^{-1}M\dual)_\sigma$. We
conclude that $M(\tau)\dual= (\tau^{-1})M\dual$.
\end{proof}

The proof of the following result works as in the un-graded case,
see \cite[Theorem 15.11, 15.12, and Corollary 15.13]{lam2}.

\begin{proposition}
Let $R$ be a graded quasi-Frobenius ring. Then the following
assertions hold.\\
$\rm (1)$ $\varphi_M$ is an isomorphism for any finitely generated
graded left $R$-module $M$.\\
$\rm (2)$ A graded left $R$ -module $M$ is finitely generated if and
only if the graded right $R$-module $M\dual$ is finitely
generated.\\
$\rm (3)$ The functor associating $M\dual$ to a graded left module
$M$ is a duality between the category of finitely generated graded
left $R$-modules and the category of finitely generated graded right
$R$-modules.\\
$\rm (4)$ If $M$ is a graded left $R$-module, then $M$ is graded
simple if and only if $M\dual$ is a graded simple right $R$-module.
\end{proposition}

We add the following characterization of graded quasi-Frobenius
algebras, whose proof goes word by word as in the un-graded case,
see \cite[Theorem 16.2]{lam2}, working with graded objects.

\begin{theorem}  \label{QFdualsimple}
Let $R$ be a graded Artinian ring. Then the following are
equivalent.\\
{\rm (i)} $R$ is graded quasi-Frobenius;\\
{\rm (ii)} The dual $S\dual$ of any graded simple left (right)
$R$-module $S$ is a graded simple right (left) $R$-module.\\
{\rm (iii)} The dual $S\dual$ of any graded simple left (right)
$R$-module $S$ is either 0 or a graded simple right (left)
$R$-module.
\end{theorem}

The following result is a graded version of \cite[Theorem
16.4]{lam2}. Moreover, its proof is on the same line as in the
ungraded case, however some new aspects occur in the graded case.
Since these will play a key role in defining graded Frobenius
rings, we sketch the proof and emphasize the parts specific to the
graded situation.

\begin{theorem} \label{QFsocleindec}
A graded Artinian ring $R$ is graded quasi-Frobenius if and only
if the following two conditions are satisfied:\\
{\rm (i)} Any graded simple left (right) $R$-module embeds up to a
shift into $R$.\\
{\rm (ii)} Any principal graded indecomposable left (right)
$R$-module has just one graded simple submodule.
\end{theorem}
\begin{proof}
Assume that conditions (i) and (ii) are satisfied. As in the
ungraded case, one shows that ${\rm soc}_\ell^{\rm gr}(R)={\rm
soc}_r^{\rm gr}(R)$.

Now we show that $S^\wedge$ is either 0 or a graded simple right
$R$-module for any graded simple left $R$-module $S$. By Lemma
\ref{lemadualsuspensie}, this is equivalent to proving it for a
shift of $S$, so by (i), we may assume that $S$ embeds into $R$.
Thus $S={\rm soc}^{\rm gr}(P)$ for some principal graded
indecomposable left $R$-module $P$. Write $P=Re$ for a homogeneous
idempotent $e$ of trivial degree. Then the map
$$\varphi :eR\ra S^\wedge , \varphi (er)(x)=xer \;\;\; \mbox{ for
any } r\in R, x\in S,$$ is a morphism of right graded $R$-modules.
Moreover, $$\varphi (eJ^{\rm gr}(R))=SeJ^{\rm gr}(R)\subseteq S
J^{\rm gr}(R)\subseteq {\rm soc}_\ell^{\rm gr}(R)J^{\rm
gr}(R)={\rm soc}_r^{\rm gr}(R)J^{\rm gr}(R)=0,$$ so $\varphi$
induces a morphism $\overline{\varphi}:eR/eJ^{\rm gr}(R)\ra
S^\wedge$.

Since $S$ is graded simple, there is an isomorphism  $\gamma:
(Re'/J^{\rm gr}(R)e')(\sigma)\ra S$ of graded left $R$-modules for
some homogeneous idempotent $e'$ of trivial degree such that $Re'$
is a principal graded indecomposable left $R$-module. As
$\overline{e'}$ (the class is modulo $J^{\rm gr}(R)e'$) has degree
$\sigma^{-1}$ in $(Re'/J^{\rm gr}(R)e')(\sigma)$, we see that
$\gamma(\overline{e'})\in S_{\sigma^{-1}}$. Clearly,
$e'\gamma(\overline{e'})=\gamma(e'\overline{e'})=\gamma(\overline{e'})$,
so then $\gamma(\overline{e'})\in e'R\cap S\subseteq e'R\cap {\rm
soc}_\ell^{\rm gr}(R)=e'R\cap {\rm soc}_r^{\rm gr}(R)={\rm
soc}^{\rm gr}(e'R)$.

Now let $u$ be a nonzero homogeneous element of $S^\wedge$. Then
$u(S)$ is isomorphic to a shift of $S$, thus it is graded simple,
and working as above for $\gamma(\overline{e'})$, we get that
$u(\gamma(\overline{e'}))\in e'R\cap u(S)\subseteq {\rm soc}^{\rm
gr}(e'R)$. Since ${\rm soc}^{\rm gr}(e'R)$ is simple, we must have
$\gamma(\overline{e'})R=u(\gamma(\overline{e'}))R={\rm soc}^{\rm
gr}(e'R)$, so $u(\gamma(\overline{e'}))=\gamma(\overline{e'})r$
for some homogeneous $r\in R$. Then for any $x\in R$ \bea
u(x\gamma(\overline{e'}))&=&x\gamma(\overline{e'})r\\
&=& x\gamma(\overline{e'})er \;\;\;\;\; (\mbox{since }
\gamma(\overline{e'})\in S\subseteq Re)\\
&=&\overline{\varphi}(\overline{er})(x\gamma(\overline{e'})),\eea
which shows that $u=\overline{\varphi}(\overline{er})$. We
conclude that $\overline{\varphi}$ is surjective. Since
$eR/eJ^{\rm gr}(R)$ is simple, this shows that $S^\wedge$ is
either 0 or isomorphic to $eR/eJ^{\rm gr}(R)$, thus simple. Now
$R$ is graded quasi-Frobenius by Theorem \ref{QFdualsimple}.
\end{proof}

A consequence of the proof of the previous theorem is the
following.

\begin{corollary} \label{prenakayamapermutation}
Let $R$ be a graded quasi-Frobenius ring. Let $e$ and $e'$ be two
primitive idempotents in $R_1$ such that ${\rm soc}^{\rm
gr}(Re)\simeq (Re'/J^{\rm gr}(R)e')(\sigma)$ as graded left
$R$-modules for some $\sigma\in G$. Then ${\rm soc}^{\rm
gr}(e'R)\simeq (\sigma) (eR/eJ^{\rm gr}(R))$ as graded right
$R$-modules.
\end{corollary}
\begin{proof}
Denote by $S={\rm soc}^{\rm gr}(Re)$, thus $S\simeq (Re'/J^{\rm
gr}(R)e')(\sigma)$. We have seen in the proof of Theorem
\ref{QFsocleindec} that $S^\wedge \simeq eR/eJ^{\rm gr}(R)$. Thus
$({\rm soc}^{\rm gr}(Re))^\wedge \simeq eR/eJ^{\rm gr}(R)$. Taking
the duals, we get ${\rm soc}^{\rm gr}(Re)\simeq (eR/eJ^{\rm
gr}(R))^\wedge$. Proceeding in a similar way with $e'$ to the right,
we see that ${\rm soc}^{\rm gr}(e'R)\simeq (Re'/J^{\rm
gr}(R)e')^\wedge$. Using Lemma \ref{lemadualsuspensie}, we obtain
that
$${\rm soc}^{\rm gr}(e'R)\simeq (S(\sigma^{-1}))^\wedge\simeq (\sigma)S^\wedge \simeq (\sigma) (eR/eJ^{\rm gr}(R))$$
\end{proof}

\begin{corollary} \label{nakayamapermutation}
Let $R$ be a graded Artinian ring, and let $Re_1,\ldots , Re_t$ be
a system of representatives for the isoshift types of principal
graded indecomposable left $R$-modules. Let $S_i=Re_i/J^{\rm
gr}(R)e_i$ and $S'_i=e_iR/e_iJ^{\rm gr}(R)$ for any
$1\leq i \leq t$. Then the following are equivalent.\\
{\rm (1)} R is graded quasi-Frobenius.\\
{\rm (2)} There exist a permutation $\pi \in S(\{ 1,\ldots,t\})$
and some $\sigma_1,\ldots ,\sigma_t\in G$ such that ${\rm
soc}^{\rm gr}(Re_i)\simeq S_{\pi (i)}(\sigma_i)$ and ${\rm
soc}^{\rm gr}(e_{\pi (i)}R)\simeq (\sigma_i)S'_i$ for any $1\leq
i\leq t$.
\end{corollary}
\begin{proof}
(1)$\Rightarrow$(2) If $R$ is graded quasi-Frobenius, then we know
by Corollary \ref{prenakayamapermutation} that there is a map
$\pi:\{ 1,\ldots , t\}\ra \{ 1,\ldots , t\}$ such that
$${\rm soc}^{\rm
gr}(Re_i)\simeq  (Re_{\pi (i)}/J^{\rm gr}(R)e_{\pi (i)})(\sigma_i)
\mbox{  and  }{\rm soc}^{\rm gr}(e_{\pi (i)}R)\simeq (\sigma)
(e_iR/e_iJ^{\rm gr}(R)).$$ We show that $\pi$ is injective,  thus
also bijective. Indeed, if $\pi (i)=\pi(j)$ for some $i,j$, then
since $Re_i$ is the injective envelope of $ {\rm soc}^{\rm
gr}(Re_i)$, we have $Re_i\simeq E^{\rm gr}(Re_{\pi (i)}/J^{\rm
gr}(R)e_{\pi (i)})(\sigma_i)$. Similarly \bea Re_j&\simeq&
E(Re_{\pi
(j)}/J^{\rm gr}(R)e_{\pi (j)})(\sigma_j)\\
&\simeq& E(Re_{\pi (i)}/J^{\rm gr}(R)e_{\pi
(i)})(\sigma_j)\\
&\simeq& (Re_i)(\sigma_j\sigma_i^{-1})\eea so $Re_i$ and $Re_j$ have
the same isoshift type, i.e., $i=j$.

(2)$\Rightarrow$(1) Since $\pi$ is bijective, we see that any left
(right) graded simple module embeds into a principal graded indecomposable left
(right) $R$-module, thus also into $R$. Moreover, the socle of any principal
graded indecomposable is graded simple by (2). We get that $R$ is graded quasi-Frobenius by Theorem \ref{QFsocleindec}.
\end{proof}

We conclude  this section by summarizing that to a graded
quasi-Frobenius ring $R$ we associate\\

$\bullet$ a positive integer $t$, the number if isoshift types of
graded principal indecomposable left $R$-module; choose some
system of representatives $P_1,\ldots,P_t$ for these isoshift
types, such that $P_1,\ldots,P_t$ embed into $R$. If
$S_i=P_i/J^{\rm gr}(R)P_i$ for $1\leq i\leq t$, we know that
$S_1,\ldots,S_t$ are the isoshift types of graded simple left
$R$-modules.

$\bullet$ some positive integers $n_1,\ldots,n_t$, indicating the
multiplicities of the isoshift types $P_1,\ldots,P_t$ in a
decomposition of $R$.

$\bullet$ a set $g_{i1},\ldots,g_{in_i}$ of elements of $G$, such
that the isoshift component of type $P_i$ of $R$ (in a decomposition
into a direct sum of graded indecomposable left modules) is
$P_i(g_{i1})\oplus \ldots \oplus P_i(g_{in_i})$.

$\bullet$ a permutation $\pi \in S(\{ 1,\ldots ,t\})$, called the
Nakayama permutation, and some elements $\sigma_1,\ldots,\sigma_t\in
G$ such that ${\rm soc}(P_i)\simeq
S_{\pi(i)}(\sigma_i)$ for any $1\leq i\leq t$.\\

Thus we associated a set of data $(t, (n_i)_{1\leq i\leq t},
(g_{ij})_{1\leq i\leq t\atop 1\leq j\leq n_i}, \pi,
(\sigma_i)_{1\leq i\leq t})$ to the graded quasi-Frobenius ring
$R$. Clearly, this set of data depends on the choices we make: the
order of the isoshift types and the choice of each $P_i$.

By the discussion at the end of Section
\ref{sectionprojectiveobjects}, we see that if we consider
$P_1=Re_1,\ldots,P_t=Re_t$ for some idempotents $e_1,\ldots,e_t\in
R_{\varepsilon}$, then $P'_1=e_1R,\ldots,P'_t=e_tR$ is a system of
representatives for the isoshift types of graded principal
indecomposable right $R$-modules, and $S'_1=P'_1/P'_1J^{\rm
gr}(R), \ldots, S'_t=P'_1/P'_tJ^{\rm gr}(R)$ is a system of
representatives for the isoshift types of graded simple right
$R$-modules. Moreover, the isoshift component of type $P'_i$ of
$R$ (as a graded right $R$-module) is $(g_{i1}^{-1})P'_i\oplus
\ldots \oplus (g_{in_i}^{-1})P'_i$ for each $i$.

If $M$ is a graded left module over a $G$-graded ring $R$, we
denote by $\Sigma (M)=\{ g\in G| M(g)\simeq M\}$, which is a
subgroup of $G$, called the inertia group of $M$. Similarly, the
inertia group of a graded right module $M$ consists of all $g$
such that $(g)M\simeq M$.

\begin{lemma} \label{calculSigmaSi}
Let $R$ be a graded quasi-Frobenius ring. With notations as above,
we have:\\
{\rm (1)} $\Sigma(S_i)=\Sigma(P_i)=\Sigma(P'_i)=\Sigma(S'_i)$ for
any
$1\leq i\leq t$;\\
{\rm (2)} $\Sigma(S_i)=\sigma_i \Sigma(S_{\pi(i)})\sigma_i^{-1}$
for any $1\leq i\leq t$.
\end{lemma}
\begin{proof}
(1) If we apply Corollary \ref{corisoproj} for $P=P_i$ and
$Q=P_i(\sigma)$, where $\sigma\in G$, we see that
$P_i(\sigma)\simeq P_i$ if and only if $S_i(\sigma)\simeq S_i$.
This shows that $\Sigma(P_i)=\Sigma(S_i)$. Similarly
$\Sigma(P'_i)=\Sigma(S'_i)$.

Now write $P_i=Re$ and $P'_i=eR$ for some idempotent $e\in
R_\varepsilon$. If $\sigma\in G$, then $\sigma\in \Sigma(P_i)$ if
and only if $Re\simeq Re(\sigma)$ in $R-gr$. By \cite[Lemma
1.2]{vas}, this is equivalent to $eR\simeq (\sigma^{-1})eR$, which
means that $\sigma^{-1}\in \Sigma(P'_i)$. Thus the subgroups
$\Sigma(P_i)$ and $\Sigma(P'_i)$ of $G$ are equal.

(2) Let $S={\rm soc}^{\rm gr}(P_i)$. We know that $S\simeq
S_{\pi(i)}(\sigma_i)$. Since $P_i$ is graded injective, we have that
$P_i=E^{\rm gr}(S)$, the injective envelope of $S$ in $R-gr$. If
$\sigma \in G$ is such that $P_i\simeq P_i(\sigma)$, then ${\rm
soc}^{\rm gr}(P_i)\simeq {\rm soc}^{\rm gr}(P_i(\sigma))$, so
$S\simeq S(\sigma)$. Conversely, if $\sigma\in G$ is such that
$S\simeq S(\sigma)$, then $E^{\rm gr}(S)\simeq E^{\rm
gr}(S(\sigma))$, so $P_i\simeq P_i(\sigma)$. Thus
$\Sigma(P_i)=\Sigma(S)=\Sigma(S_{\pi(i)}(\sigma_i))=\sigma_i
\Sigma(S_{\pi(i)})\sigma_i^{-1}$, and the result follows since
$\Sigma(S_i)=\Sigma(P_i)$.
\end{proof}

\section{Graded Frobenius rings}  \label{sectionFrobenius}

We recall (see \cite[Theorem 16.14 and Corollary 16.16]{lam2})
that a two-sided Artinian ring $R$ is called Frobenius if it
satisfies one of the following equivalent conditions: (1) $R$ is
quasi-Frobenius and ${\rm soc}_\ell(R)\simeq R/J(R)$ as left
$R$-modules; (2) $R$ is quasi-Frobenius and ${\rm soc}_r(R)\simeq
R/J(R)$ as right $R$-modules; (3) $R$ is quasi-Frobenius and
$n_i=n_{\pi(i)}$ for any $1\leq i\leq t$, where $n_1,\ldots,n_t$
are the multiplicities of the isomorphism types of principal
indecomposable modules, and $\pi$ is the Nakayama permutation
(which is just the one we described in Section
\ref{sectionmoreQF}, when we regard $R$ as a graded ring with
trivial grading); (4) ${\rm soc}_\ell(R)\simeq R/J(R)$ as left
$R$-modules, and ${\rm soc}_r(R)\simeq R/J(R)$ as right
$R$-modules; (5) $R$ is quasi-Frobenius and $R/J(R)\simeq {\rm
Hom}_{-R}(R/J(R),R)$ as left $R$-modules; (6) $R$ is
quasi-Frobenius and $R/J(R)\simeq {\rm Hom}_{R-}(R/J(R),R)$ as
right $R$-modules.

We will need the following simple fact.

\begin{lemma}  \label{lemasumesimple}
If $R$ is a $G$-graded ring, $S$ is a graded simple left
$R$-module, $S'$ is a graded simple right $R$-module, and
$g_1,\ldots,g_n,h_1, \ldots,h_n\in G$, then:\\
{\rm (1)} $S(g_1)\oplus \ldots \oplus S(g_n)\simeq S(h_1)\oplus
\ldots \oplus S(h_n)$ in $R-gr$ if and only if the sequence of
left $\Sigma(S)$-cosets $g_1\Sigma(S),\ldots,g_n\Sigma(S)$ is a
permutation of $h_1\Sigma(S),\ldots,h_n\Sigma(S)$.\\
{\rm (2)} $(g_1)S'\oplus \ldots \oplus (g_n)S'\simeq (h_1)S'\oplus
\ldots \oplus (h_n)S'$ in $gr-R$ if and only if the sequence of
right cosets $\Sigma(S')g_1,\ldots,\Sigma(S')g_n$ is a permutation
of $\Sigma(S')h_1,\ldots,\Sigma(S')h_n$.
\end{lemma}
\begin{proof}
(1) follows from the fact that $S(g)\simeq S(h)$ if and only if
$g^{-1}h\in \Sigma (S)$, i.e., $g\Sigma(S)=h\Sigma(S)$. (2) is
similar.
\end{proof}

Let $R$ be a graded Artinian ring. Let $P_1,\ldots,P_t$ be a
system of representatives for the isoshift types of principal
graded indecomposable left $R$-modules, and say that the principal
graded indecomposable left $R$-modules of isoshift type $P_i$ that
occur in a decomposition of $R$ are isomorphic to
$P_i(g_{i1}),\ldots ,P_i(g_{in_i})$ for any $1\leq i\leq t$.

Then $S_i=P_i/J^{\rm gr}(R)P_i$, $1\leq i\leq t$, is a system of
representatives for the isoshift types of graded simple left
$R/J^{\rm gr}(R)$-modules, thus also for the isoshift types of
graded simple left $R$-modules, and a decomposition of $R/J^{\rm
gr}(R)$ as a sum of graded simple left $R$-modules is $$R/J^{\rm
gr}(R)=\oplus_{1\leq i\leq t}(\oplus_{1\leq j\leq
n_i}S_i(g_{ij})).$$

If moreover $R$ is graded quasi-Frobenius, let
$\sigma_1,\ldots,\sigma_t\in G$ be such that ${\rm soc}^{\rm
gr}(P_i)\simeq S_{\pi (i)}(\sigma_i)$ for any $i$, where $\pi \in
S(\{ 1,\ldots, t\})$ is the Nakayama permutation associated with
$R$ .

Recall that if $R$ is a graded quasi-Frobenius ring, then ${\rm
soc}_\ell^{\rm gr}(R)={\rm soc}_r^{\rm gr}(R)$, and we simply denote
this graded ideal of $R$ by ${\rm soc}^{\rm gr}(R)$.

\begin{theorem} \label{teoremagradedFrobenius}
Let $R$ be a graded Artinian ring, and let $\sigma\in G$. The following are equivalent.\\
$\mathrm{(1)}$ $R$ is graded quasi-Frobenius and
${\rm soc}^{\rm gr}(R)(\sigma)\simeq R/J^{\rm gr}(R)$ in $R-gr$.\\
$\mathrm{(2)}$ $R$ is graded quasi-Frobenius and
$(\sigma){\rm soc}^{\rm gr}(R)\simeq R/J^{\rm gr}(R)$ in $gr-R$.\\
$\mathrm{(3)}$ $R$ is graded quasi-Frobenius and for any $1\leq
i\leq t$ we have $n_i=n_{\pi (i)}$, and the sequence of left
cosets $\sigma g_{i1}\sigma_i\Sigma (S_{\pi(i)}), \ldots ,\sigma
g_{in_i}\sigma_i\Sigma (S_{\pi(i)})$ is a permutation of  $ g_{\pi
(i)1}\Sigma (S_{\pi(i)}), \ldots , g_{\pi
(i)n_i}\Sigma (S_{\pi(i)})$.\\
$\mathrm{(4)}$  ${\rm soc}_\ell^{\rm gr}(R)(\sigma)\simeq R/J^{\rm
gr}(R)$ in $R-gr$ and
$(\sigma){\rm soc}_r^{\rm gr}(R)\simeq R/J^{\rm gr}(R)$ in $gr-R$.\\
$\mathrm{(5)}$ $R$ is graded quasi-Frobenius and $(R/J^{\rm
gr}(R))\dual (\sigma)\simeq R/J^{\rm gr}(R)$ in $R-gr$ (where
$R/J^{\rm gr}(R)$ is regarded as a graded right $R$-module in the
left hand side, and as a graded left $R$-module in the right hand
side).\\
$\mathrm{(6)}$ $R$ is graded quasi-Frobenius and
$(\sigma)(R/J^{\rm gr}(R))\dual \simeq R/J^{\rm gr}(R)$ in $gr-R$
(where $R/J^{\rm gr}(R)$ is regarded as a graded left $R$-module
in the left hand side, and as a graded right $R$-module in the
right hand side).
\end{theorem}
\begin{proof}
(1)$\Leftrightarrow$(3) The decomposition $R=\oplus_{1\leq i\leq
t}\oplus_{1\leq j\leq n_i}P_{ij}$ shows that
 \bea {\rm soc}^{\rm gr}(R)&\simeq&
\oplus_{1\leq i\leq t}(\oplus_{1\leq j\leq
n_i}{\rm soc}^{\rm gr}P_i(g_{ij})) \\
&\simeq& \oplus_{1\leq i\leq t}(\oplus_{1\leq j\leq
n_i}(S_{\pi (i)}(\sigma_i))(g_{ij}))\\
&\simeq& \oplus_{1\leq i\leq t}(\oplus_{1\leq j\leq
n_i}(S_{\pi (i)}(g_{ij}\sigma_i)))
\eea

On the other hand, $R/J^{\rm gr}(R)=\oplus_{1\leq i\leq
t}(\oplus_{1\leq j\leq n_i}S_i(g_{ij}))$, so then ${\rm soc}^{\rm
gr}(R)(\sigma)\simeq R/J^{\rm gr}(R)$ if and only if the
components of the same isoshift type in these two graded
$R$-modules are isomorphic, i.e.,
$$\oplus_{1\leq j\leq
n_i}(S_{\pi (i)}(\sigma g_{ij}\sigma_i))\simeq \oplus_{1\leq j\leq
n_{\pi (i)}}S_{\pi (i)}(g_{\pi (i)j})$$ for any $i$. Using Lemma
\ref{lemasumesimple}, this is equivalent to the fact that for any
$i$ we have $n_i=n_{\pi (i)}$, and the sequence of left cosets
$\sigma g_{i1}\sigma_i\Sigma (S_{\pi(i)}), \ldots ,\sigma
g_{in_i}\sigma_i\Sigma (S_{\pi(i)})$ is a permutation of $ g_{\pi
(i)1}\Sigma (S_{\pi(i)}), \ldots , g_{\pi (i)n_i}\Sigma
(S_{\pi(i)})$.

(2)$\Leftrightarrow$(3) As in (1)$\Leftrightarrow$(3), the
decomposition $R=\oplus_{1\leq r\leq t}\oplus_{1\leq j\leq
n_r}P'_{rj}$ shows that

$${\rm soc}^{\rm gr}(R)\simeq \oplus_{1\leq r\leq t}(\oplus_{1\leq j\leq
n_r}((\sigma_{\pi^{-1}(r)}g_{rj}^{-1})S'_{\pi^{-1} (r)}))$$

Hence $(\sigma){\rm soc}^{\rm gr}(R)\simeq R/J^{\rm gr}(R)$ if and
only if
$$\oplus_{1\leq j\leq n_r}(\sigma_{\pi^{-1}(r)}g_{rj}^{-1}\sigma)S'_{\pi^{-1} (r)}\simeq
\oplus_{1\leq j\leq n_{\pi^{-1} (r)}}(g_{\pi^{-1}(r)j}^{-1}
)S'_{\pi^{-1} (r)}$$ for any $r$. Denoting $i=\pi^{-1}(r)$, this
means that $n_i=n_{\pi(i)}$ and, using
 Lemma \ref{lemasumesimple}, that
 $\Sigma(S'_i)g_{i1}^{-1},\ldots,\Sigma(S'_i)g_{in_i}^{-1}$ is a
 permutation of $\Sigma(S'_i)\sigma_i g_{\pi
 (i)1}^{-1}\sigma,\ldots,\Sigma(S'_i)\sigma_ig_{\pi(i)n_i}^{-1}\sigma$.
 Passing to left cosets and taking into account that
 $\Sigma(S'_i)=\Sigma(S_i)$, this rewrites that
 $g_{i1}\Sigma(S_i), \ldots,g_{in_i}\Sigma(S_i)$ is a permutation
 of
 $\sigma^{-1}g_{\pi(i)1}\sigma_i^{-1}\Sigma(S_i),\ldots,\sigma^{-1}g_{\pi(i)n_i}\sigma_i^{-1}\Sigma(S_i)$.
 By Lemma \ref{calculSigmaSi} we know that $\Sigma(S_i)=\sigma_i
 \Sigma(S_{\pi(i)})\sigma_i^{-1}$, and the condition becomes that $$g_{i1}\sigma_i\Sigma(S_{\pi(i)})\sigma_i^{-1},
 \ldots,g_{in_i}\sigma_i\Sigma(S_{\pi(i)})\sigma_i^{-1}$$ is a permutation
 of
 $$\sigma^{-1}g_{\pi(i)1}\Sigma(S_{\pi(i)})\sigma_i^{-1},\ldots,\sigma^{-1}g_{\pi(i)n_i}\Sigma(S_{\pi(i)})\sigma_i^{-1},$$
 which after right multiplication by $\sigma_i$ and left multiplication by $\sigma$ becomes just the
 condition in (3).

 It is obvious that if (1) holds (thus so does (2)), then (4)
 holds, too. Assume now that (4) holds and we prove (1). In fact, we just need to show that $R$ is graded
 quasi-Frobenius. Since each isoshift type of graded simple
 left $R$-modules occurs inside $R/J^{\rm gr}(R)$, it also occurs
 inside ${\rm soc}_\ell^{\rm gr}(R)$, thus it embeds into $R$, too. On the
 other hand, if we decompose $R=\oplus_{1\leq p\leq m}Q_p$ into a sum
 of indecomposable graded left $R$-modules, then $R/J^{\rm gr}(R)\simeq \oplus_{1\leq p\leq
 m}Q_p/J^{\rm gr}(R)Q_p$, a direct sum of $m$ graded simple modules,
 so then
 ${\rm soc}_\ell^{\rm gr}(R)(\sigma)=\oplus_{1\leq p\leq m}{\rm soc}^{\rm gr}(Q_p)(\sigma)$ is also a direct sum of $m$ graded simple modules.
 As each ${\rm soc}^{\rm gr}(Q_p)$ is non-zero, we get that it
 must be
 graded simple. We proceed the same to the right and Theorem
 \ref{QFsocleindec} shows that $R$ is graded quasi-Frobenius.

(1)$\Leftrightarrow$(5) Let us first note that if $R$ is graded
quasi-Frobenius, then the injectivity of $R$ shows that for any
morphism $f:{\rm soc}^{\rm gr}(R)\ra R$ of degree $\sigma$ of graded
left $R$-modules, there exists $a\in R_\sigma$ such that $f(r)=ra$
for any $r\in {\rm soc}^{\rm gr}(R)$; denote by $f_a$ the morphism
associated with $a$ in this way. Clearly, $f_a=f_b$ if and only if
$a-b\in {\rm ann}_r({\rm soc}^{\rm gr}(R))=J^{\rm gr}(R)$. This
induces an isomorphism of graded right $R$-modules $R/J^{\rm
gr}(R)\simeq {\rm soc}^{\rm gr}(R)\dual$, which associates $f_a$ to
$\hat{a}\in R/J^{\rm gr}(R)$ ($\hat{a}$ is the class of $a$ modulo
$J^{\rm gr}(R)$). Therefore $(R/J^{\rm gr}(R))\dual\simeq {\rm
soc}^{\rm gr}(R)\dual\dual\simeq {\rm soc}^{\rm gr}(R)$ as graded
left $R$-modules. Now the equivalence of the two conditions is
clear.

(2)$\Leftrightarrow$(6) is similar to (1)$\Leftrightarrow$(5),
working the opposite side.
\end{proof}

A graded Artinian ring $R$ satisfying the equivalent conditions in
Theorem \ref{teoremagradedFrobenius} will be called a {\it
$\sigma$-graded Frobenius ring}. An $\varepsilon$-graded Frobenius
ring will be simply called a {\it graded Frobenius ring}.

The following gives some examples of graded Frobenius rings and a
method of constructing new graded Frobenius rings from known ones.

\begin{proposition}
{\rm (1)} If $R_1,\ldots,R_n$ are $G$-graded rings, then
$R_1\times \ldots \times R_n$ is graded quasi-Frobenius if and
only if
$R_1,\ldots,R_n$ are graded quasi-Frobenius. \\
{\rm (2)} If $R_1,\ldots,R_n$ are $G$-graded rings and $\sigma\in
G$, then $R_1\times \ldots \times R_n$ is $\sigma$-graded
Frobenius if and only if
$R_1,\ldots,R_n$ are $\sigma$-graded Frobenius. \\
{\rm (3)} A graded semisimple ring $R$ is graded Frobenius. In
particular, graded division rings are graded Frobenius.
\end{proposition}
\begin{proof}
(1) follows easily if we use the characterization of the graded
quasi-Frobenius property in Theorem \ref{gradedQF} (1).

(2) follows from the characterization of the graded Frobenius
property in Theorem \ref{teoremagradedFrobenius} (1).

(3) $R$ is graded Artinian, thus graded left Noetherian, since it
is graded semisimple. Moreover, $R-gr$ is a semisimple category,
so any object is injective. In particular $R$ is injective as a
graded left $R$-module. Thus $R$ is graded quasi-Frobenius. Now
${\rm soc}^{\rm gr}(R)=R$ and $J^{\rm gr}(R)=0$, so obviously
${\rm soc}^{\rm gr}(R)\simeq R/J^{\rm gr}(R)$ in $R-gr$.
\end{proof}

Now we show that the concept of a graded Frobenius ring matches
with the one of a graded Frobenius algebra in the case of a finite
dimensional graded algebra over a field. Let $A=\oplus_{g\in
G}A_g$ be a finite dimensional $G$-graded $k$-algebra, where $k$
is a field. Then the linear dual space $A^*=Hom_k(A,k)$ is a
$G$-graded vector space, whose homogeneous component of degree $g$
is $(A^*)_g=\{ f\in A^*|\; f(A_h)=0\;\mbox{for any }h\neq
g^{-1}\}$. Moreover, when regarded with the $A$-bimodule structure
induced by the $A$-bimodule structure of $A$, $A^*$ becomes a
graded left $A$-module and a graded right $A$-module. Then $A$ is
called a $\sigma$-graded Frobenius algebra if $A(\sigma)\simeq
A^*$ in $A-gr$, or equivalently, if $(\sigma)A\simeq A^*$ in
$gr-A$, see \cite[Section 3]{dnn1}.

\begin{proposition}
Let $R$ be a finite dimensional $G$-graded $k$-algebra, where $k$
is a field, and let $\sigma\in G$. Then $R$ is a $\sigma$-graded
Frobenius algebra if and only if it is a $\sigma$-graded Frobenius
ring.
\end{proposition}
\begin{proof}
Regard $R^*$ as a graded left $R$-module. Since $R$ is left
Artinian, we use Remark \ref{socluanulator} to see that $${\rm
soc}^{\rm gr}(_RR^*)=\{ r^*\in R^*|J^{\rm gr}(R)r^*=0\}=\{ r^*\in
R^*|r^*(J^{\rm gr}(R))=0\}.$$ This shows that there is an
isomorphism ${\rm soc}^{\rm gr}(_RR^*)\simeq (R/J^{\rm gr}(R))^*$
of graded left $R$-modules. Now $R/J^{\rm gr}(R)$ is a finite
dimensional graded semisimple $k$-algebra, so by \cite[Corollary
4.5]{dnn1} it is graded Frobenius. Thus $(R/J^{\rm gr}(R))^*\simeq
R/J^{\rm gr}(R)$ as graded left $R/J^{\rm gr}(R)$-modules, and
then also as graded left $R$-modules. We conclude that ${\rm
soc}^{\rm gr}(_RR^*)\simeq R/J^{\rm gr}(R)$ as graded left
$R$-modules.

Let us also note that $R^*$ is an injective left $R$-module (since
$R$ is a projective right $R$-module), and then so is as a graded
left $R$-module.

Now assume that $R$ is a $\sigma$-graded Frobenius algebra, i.e.,
$R(\sigma)\simeq R^*$ in $R-gr$. Thus $R(\sigma)$ is injective as
a graded left $R$-module, and then so is $R$, showing that $R$ is
graded quasi-Frobenius. Moreover, taking the socles, we have that
${\rm soc}^{\rm gr}_\ell(R)(\sigma)\simeq {\rm soc}^{\rm
gr}(_RR^*)\simeq R/J^{\rm gr}(R)$ as graded left $R$-modules. This
shows that $R$ is a $\sigma$-graded Frobenius ring.

Conversely, if $R$ is a $\sigma$-graded Frobenius ring, then $R$
is graded quasi-Frobenius, so $R$ is an injective graded left
$R$-module. Since $R$ is finite dimensional, ${\rm soc}^{\rm
gr}_\ell(R)$ is an essential graded left submodule of $R$, so then
$R=E^{\rm gr}({\rm soc}^{\rm gr}_\ell(R))$. On the other hand,
${\rm soc}^{\rm gr}(_RR^*)$ is essential in $R^*$ (since $R^*$ is
finite dimensional), and $_RR^*$ is injective, so $R^*=E^{\rm
gr}({\rm soc}^{\rm gr}(_RR^*))\simeq E^{\rm gr}(R/J^{\rm gr}(R))$.
As ${\rm soc}^{\rm gr}_\ell(R)(\sigma)\simeq R/J^{\rm gr}(R)$, we
get that $R(\sigma)\simeq R^*$ in $R-gr$, so $R$ is a graded
Frobenius algebra.
\end{proof}

We will give a structure result for $\sigma$-graded Frobenius
rings. We recall that if $M$ is a graded left $R$-module, and
$\sigma\in G$, then $M$ is called $\sigma$-faithful if
$X_\sigma\neq 0$ for any non-zero graded submodule $X$ of $M$;
this is equivalent to $R_{\sigma g^{-1}}m_g\neq 0$ for any
non-zero homogeneous element $m_g\in M_g$. The $\sigma$-faithful
condition can be defined similarly for graded right modules. The
following is obvious.

\begin{lemma} \label{lemafaithful}
Let $M$ be a graded left $R$-module, and let $U$ be a graded
submodule of $M$. The following hold.\\
{\rm (i)} If $M$ is $\sigma$-faithful, then so is $U$.\\
{\rm (ii)} If $U$ is essential in $M$ and $U$ is
$\sigma$-faithful, then $M$ is $\sigma$-faithful.\\
{\rm (iii)} If $U$ is essential in $M$, and $U$ is graded
semisimple, then $U={\rm soc}^{\rm gr}(M)$.
\end{lemma}

We say that the graded ring $R$ is $\sigma$-faithful to the left
(right) if it is $\sigma$-faithful as a graded left (right)
$R$-module.

\begin{lemma} \label{lemasoclu}
Let $R$ be a graded left Artinian graded ring which is
$\sigma$-faithful to the left. Then ${\rm soc}^{\rm gr}_\ell
(R)_\sigma={\rm soc}(_{R_\varepsilon}R_\sigma)$.
\end{lemma}
\begin{proof}
"$\subset$" Let $\Sigma$ be a graded simple left submodule of $R$.
Then $\Sigma_\sigma$ is either 0 or a simple
$R_\varepsilon$-module. In either case $\Sigma_\sigma\subset {\rm
soc}(_{R_\varepsilon}R_\sigma)$.

"$\supset$" Let $S$ be a simple $R_\varepsilon$-submodule of
$R_\sigma$. Then $RS$ is a non-zero graded left submodule of $R$.
As $R$ is graded left Artinian, $U={\rm soc}^{\rm gr}(RS)\neq 0$.
Since $R$ is $\sigma$-faithful to the left, $U_\sigma\neq 0$. But
$U_\sigma \subset (RS)_\sigma=S$, so $U_\sigma=S$. Then
$S=U_\sigma\subset {\rm soc}(_{R_\varepsilon}R_\sigma)$.
\end{proof}

\begin{lemma} \label{1fidel}
Let $R$ be a graded left Artinian ring. Then the graded left
$R$-module $R/J^{\rm gr}(R)$ is $\varepsilon$-faithful.
\end{lemma}
\begin{proof}
The graded $R$-submodules of $R/J^{\rm gr}(R)$ are the same as the
graded $R/J^{\rm gr}(R)$-submodules of $R/J^{\rm gr}(R)$. Now the
result follows since $R/J^{\rm gr}(R)$ is a graded semisimple
ring, and any graded semisimple ring $A$ is $\varepsilon$-faithful
to the left (and to the right) by \cite[Proposition 2.9.6]{nvo}.
\end{proof}

We also recall the definition of the coinduced functor from
\cite[Section 2.5]{nvo}. If $N$ is a left $R_\varepsilon$-module,
denote $Coind(N)_g=\{ f\in Hom_{R_\varepsilon}(R,N)| f(R_h)=0
\mbox{ for any } h\neq g^{-1}\}$ for each $g\in G$. Then
$\sum_{g\in G}Coind(N)_g$ is a direct sum inside
$Hom_{R_\varepsilon}(R,N)$, which we denote by $Coind(N)$. Then
$Coind(N)$ is an $R$-submodule of $Hom_{R_\varepsilon}(R,N)$,
moreover, it is a graded $R$-module with the decomposition given
by the sum above. If $M$ is a graded left $R$-module, and
$\sigma\in G$, then the map
$$\nu_M: M\ra Coind(M_\sigma)(\sigma^{-1}), \nu_M(m_g)(a)=a_{\sigma g^{-1}}m_g \mbox{  for any }m_g\in M_g, a\in
R$$ is a morphism of graded $R$-modules, see \cite[page 39]{nvo}.
Moreover, ${\rm Im}\,\nu_M$ is an essential submodule of
$Coind(M_\sigma)(\sigma^{-1})$ (\cite[Proposition 2.6.2]{nvo}),
and $\nu_M$ is injective if $M$ is $\sigma$-faithful
(\cite[Proposition 2.6.3]{nvo}).

\begin{theorem}
Let $R$ be a graded Artinian ring, and let $\sigma \in G$. The
following are equivalent.\\
{\rm (1)} $R$ is a $\sigma$-graded Frobenius ring.\\
{\rm (2)} $R$ is $\sigma$-faithful to the left and to the right,
${\rm soc}(_{R_{\varepsilon}}R_\sigma)\simeq
R_\varepsilon/J(R_\varepsilon)$ as left $R_\varepsilon$-modules,
and ${\rm soc}((R_\sigma)_{R_{\varepsilon}})\simeq
R_\varepsilon/J(R_\varepsilon)$ as right $R_\varepsilon$-modules.
\end{theorem}
\begin{proof}

(1)$\Rightarrow$(2) Since $R$ is graded left Artinian, ${\rm
soc}^{\rm gr}_\ell(R)$ is essential in $R$ as a graded left
submodule. As $R$ is graded injective, we have $E^{\rm gr}({\rm
soc}^{\rm gr}_\ell(R))=R$. Now ${\rm soc}^{\rm gr}_\ell(R)\simeq
(R/J^{\rm gr}(R))(\sigma^{-1})$, and we get $R=E^{\rm gr}({\rm
soc}^{\rm gr}_\ell(R))\simeq E^{\rm gr}(R/J^{\rm
gr}(R))(\sigma^{-1})$. By Lemma \ref{1fidel}, $R/J^{\rm gr}(R)$ is
$\varepsilon$-faithful as a graded left $R$-module, then so is
$E^{\rm gr}(R/J^{\rm gr}(R))$ by Lemma \ref{lemafaithful}(ii).
Therefore its shift $E^{\rm gr}(R/J^{\rm gr}(R))(\sigma^{-1})$ is
$\sigma$-faithful, showing that $R$ is $\sigma$-faithful to the
left.

The isomorphism ${\rm soc}^{\rm gr}_\ell(R)(\sigma)\simeq R/J^{\rm
gr}(R)$ of graded left $R$-modules, induces an isomorphism of left
$R_\varepsilon$-modules between the homogeneous components of
degree $\varepsilon$. We have $({\rm soc}^{\rm
gr}_\ell(R)(\sigma))_\varepsilon={\rm soc}^{\rm
gr}_\ell(R)_{\sigma}={\rm soc}(_{R_{\varepsilon}}R_\sigma)$ (the
last equality following from Lemma \ref{lemasoclu}), and
$(R/J^{\rm gr}(R))_\varepsilon =R_\varepsilon/J(R_\varepsilon)$
since $J^{\rm gr}(R)\cap R_{\varepsilon}=J(R_\varepsilon)$ (see
\cite[Corollary 2.9.3]{nvo}). We get ${\rm
soc}(_{R_{\varepsilon}}R_\sigma)\simeq
R_\varepsilon/J(R_\varepsilon)$ as left $R_\varepsilon$-modules.
Working similarly to the right, we obtain that $R$ is
$\varepsilon$-faithful to the right and ${\rm
soc}((R_\sigma)_{R_{\varepsilon}})\simeq
R_\varepsilon/J(R_\varepsilon)$ as right $R_\varepsilon$-modules.

(2)$\Rightarrow$(1) Since $R$ is $\sigma$-faithful to the left,
${\rm soc}^{\rm gr}_\ell (R)$ is also $\sigma$-faithful.
Therefore, taking into account the above considerations and Lemma
\ref{lemasoclu}, we see that
$$\nu_{{\rm soc}^{\rm gr}_\ell (R)}:{\rm soc}^{\rm gr}_\ell (R)\ra
Coind({\rm soc}^{\rm gr}_\ell (R)_\sigma)(\sigma^{-1})= Coind({\rm
soc}(_{R_{\varepsilon}}R_\sigma))(\sigma^{-1})$$ is an essential
injective morphism in $R-gr$. As ${\rm soc}^{\rm gr}_\ell (R)$ is
graded semisimple, it follows by Lemma \ref{lemafaithful}(iii)
that ${\rm soc}^{\rm gr}_\ell (R)(\sigma)\simeq ({\rm Im}\,
\nu_{{\rm soc}^{\rm gr}_\ell (R)})(\sigma)={\rm soc}^{\rm
gr}(Coind({\rm soc}(_{R_{\varepsilon}}R_\sigma)))$.

On the other hand, $R/J^{\rm gr}(R)$ is $\varepsilon$-faithful as
a graded left $R$-module by Lemma \ref{1fidel}, so
$$\nu_{R/J^{\rm gr}(R)}:R/J^{\rm gr}(R)\ra Coind ((R/J^{\rm
gr}(R))_\varepsilon)=Coind(R_\varepsilon/J(R_\varepsilon))$$ is an
essential injective morphism in $R-gr$. Since $R/J^{\rm gr}(R)$ is
a semisimple graded $R$-module, we see again by Lemma
\ref{lemafaithful}(iii) that
$$R/J^{\rm gr}(R)\simeq {\rm Im}\,\nu_{R/J^{\rm gr}(R)}={\rm
soc}^{\rm gr}(Coind(R_\varepsilon/J(R_\varepsilon))).$$ Now ${\rm
soc}(_{R_{\varepsilon}}R_\sigma)\simeq
R_\varepsilon/J(R_\varepsilon)$ as left $R_\varepsilon$-modules,
therefore  $Coind({\rm soc}(_{R_{\varepsilon}}R_\sigma))\simeq
Coind(R_\varepsilon/J(R_\varepsilon))$ are isomorphic as graded
left $R$-modules, and then so are their socles. We conclude that
${\rm soc}^{\rm gr}_\ell(R)(\sigma)\simeq R/J^{\rm gr}(R)$.
Working similarly to the right (with the adapted version of the
coinduced functor), we obtain $(\sigma){\rm soc}_r^{\rm
gr}(R)\simeq R/J^{\rm gr}(R)$ in $gr-R$. These show that $R$ is
$\sigma$-graded Frobenius.

\end{proof}

\begin{corollary} \label{corolargradedFrobenius}
Let $R$ be a graded Artinian ring. Then $R$ is graded Frobenius if
and only if it is $\varepsilon$-faithful to the left and to the
right and $R_\varepsilon$ is a Frobenius ring.
\end{corollary}

\begin{remark}
{\rm The previous Corollary shows that If $R$ is a graded
Frobenius ring, then its homogeneous component $R_\varepsilon$ of
trivial degree is a Frobenius ring. We note that a similar
transfer does not hold for the quasi-Frobenius property, more
precisely, $R$ may be graded quasi-Frobenius, such that
$R_\varepsilon$ is not quasi-Frobenius.

We first recall from \cite[Example 16.60]{lam2} that if $A$ is a
finite dimensional algebra over a field $k$, then
$\mathcal{E}(A)=A\oplus A^*$ has a $k$-algebra structure with
multiplication defined by $(a,a^*)(b,b^*)=(ab, ab^*+a^*b)$ for any
$a,b\in A$ and $a^*,b^*\in A^*$; here we regard $A^*$ as an
$A$-bimodule in the usual way. $\mathcal{E}(A)$ is called the
trivial extension of $A$, and it is always a Frobenius algebra
(even a symmetric algebra). Moreover, $\mathcal{E}(A)$ has a
grading by the cyclic group $C_2=\{ \varepsilon, c\}$ of order 2,
with $\mathcal{E}(A)_\varepsilon=A\oplus 0$ and
$\mathcal{E}(A)_c=0\oplus A^*$. One can easily see that
$\mathcal{E}(A)$ is not $\varepsilon$-faithful to the left, but it
is $c$-faithful to the left.

Now we see that if $A$ is a finite dimensional $k$-algebra which
is not quasi-Frobenius, then $\mathcal{E}(A)$ is a Frobenius ring
(since it is a Frobenius algebra), so it is a quasi-Frobenius
ring, thus also a graded quasi-Frobenius ring. On the other hand,
$\mathcal{E}(A)_\varepsilon\simeq A$ is not a quasi-Frobenius
ring.

Now if we take a Frobenius finite dimensional $k$-algebra $A$,
then we see that $\mathcal{E}(A)_\varepsilon$ is a Frobenius ring,
but $\mathcal{E}(A)$ is not a graded Frobenius ring, since it is
not $\varepsilon$-faithful to the left. Thus the "if" implication
in Corollary \ref{corolargradedFrobenius} does not hold anymore if
we omit the $\varepsilon$-faithful condition.

The next result shows that these connections work better for
strongly graded rings.}
\end{remark}

\begin{proposition}
Let $R$ be a strongly graded ring. The following assertions
hold.\\
{\rm (1)} $R$ is a graded quasi-Frobenius ring if and only if
$R_\varepsilon$ is a quasi-Frobenius ring.\\
{\rm (2)} $R$ is a graded Frobenius ring if and only if
$R_\varepsilon$ is a Frobenius ring.
\end{proposition}
\begin{proof}
As a consequence of Dade's Theorem, which says that the induced
functor $R\ot_{R_\varepsilon}-:R_{\varepsilon}-mod \ra R-gr$ is an
equivalence of categories (see \cite[Theorem 3.1.1]{nvo}), we have
that $R$ is graded left Artinian (Noetherian) if and only if
$R_\varepsilon$ is left Artinian (Noetherian), and the same fact
is true to the right. Also, $R$ is injective in $R-gr$ if and only
if $R_\varepsilon$ is injective in $R_\varepsilon$-mod. Now (1) is
clear by using Theorem \ref{gradedQF}.

On the other hand, $R$ is $\varepsilon$-faithful to the left (and
to the right). Indeed, let $g\in G$ and $r_g\in R_g$ such that
$R_{g^{-1}}r_g=0$. Then $R_gR_{g^{-1}}r_g=0$. But $R$ is strongly
graded, so $R_gR_{g^{-1}}=R_\varepsilon$, so $R_\varepsilon
r_g=0$, showing that $r_g=0$.  Now (2) follows directly from
Corollary \ref{corolargradedFrobenius}.
\end{proof}

\begin{remark}
{\rm If $R$ is a finite dimensional graded algebra over a field,
it is obvious that if $R$ is graded Frobenius, then $R$ is a
Frobenius algebra; indeed, we just regard an isomorphism of graded
left $R$-modules between $R$ and $R^*$ just as an isomorphism of
$R$-modules.

If $R$ is a $G$-graded ring which is graded Frobenius, then $R$ is
not necessarily a Frobenius ring. Indeed, we can use the example
after Remark \ref{definitiegradedQF}: $R=k[X,X^{-1}]$ is a graded
division ring, so it is clearly a graded Frobenius ring, while it
is not even quasi-Frobenius.

We do not know whether for a finite group $G$, a graded Frobenius
ring $R$ is also a Frobenius ring. This is true in the case where
the order of $G$ is invertible in $R$. Indeed, under this
condition, a consequence of graded Clifford theory is that $J^{\rm
gr}(R)=J(R)$ and ${\rm soc}^{\rm gr}(R)={\rm soc}(R)$, see
\cite[Corollary 4.4.5 and Proposition 4.4.10]{nvo}. Since $R$ is
graded Frobenius, it is graded quasi-Frobenius, thus also
quasi-Frobenius, and then the result is clear taking into account
Theorem \ref{teoremagradedFrobenius} (1).}
\end{remark}

\end{document}